\documentclass[11pt]{amsart}
\usepackage{amscd}
\usepackage{amsmath,amssymb,latexsym}
\usepackage{pb-diagram}
\usepackage{enumerate}
\usepackage{amsfonts}
\usepackage{appendix}
\usepackage{hyperref}
\usepackage{cleveref}

\makeatletter
\def\imod#1{\allowbreak\mkern10mu({\operator@font mod}\,\,#1)}
\makeatother

\makeatletter
\renewcommand\section{\@startsection{section}{1}{\z@}%
                                  {-3.5ex \@plus -1ex \@minus-.2ex}%
                                  {2.3ex \@plus.2ex}%
                                  {\center\normalfont\large\bfseries}}
\makeatother

\theoremstyle{plain}
\newtheorem{Conj}{Conjecture}[section]
\newtheorem{Thm}[Conj]{Theorem}
\newtheorem{Cor}[Conj]{Corollary}
\newtheorem{Prop}[Conj]{Proposition}
\newtheorem{Lem}[Conj]{Lemma}

\theoremstyle{definition}
\newtheorem{Remark}[Conj]{Remark}

\numberwithin{equation}{section}
\newcommand{\integral}[1]{\int\limits_{ {#1}\left(F\right)\backslash {#1}\left(\A\right)}}
\newcommand{\E}{\mathcal E}
\newcommand{\C}{\mathbb C}
\newcommand{\bH}{\mathbb H}
\newcommand{\A}{\mathbb A}
\newcommand{\G}{\mathbb G}
\newcommand{\mA}{\mathcal A} 

\newcommand{\Sch}{\mathcal S}  
\newcommand{\Z}{\mathbb Z}

\newcommand\ind{\operatorname{ind}}
\newcommand\Ind{\operatorname{Ind}}

\newcommand\Hom{\operatorname{Hom}}
\newcommand\Res{\operatorname{Res}}

\newcommand{\TSL}{{\widetilde{SL_2}}}
\newcommand{\diag}{\operatorname{diag}}

\DeclareMathOperator{\zfun}{\zeta} %Zeta-function
\DeclareMathOperator{\zint}{\mathcal{Z}} %Zeta-integral
\DeclareMathOperator{\Lfun}{\mathcal{L}} %$\Lfun$-functions
\newcommand{\coset}[1]{\left[ #1 \right]}  %square brackets
\newcommand{\FNorm}[1]{\left\vert #1 \right\vert} %Norm in a local field/absolute value
\newcommand{\Nm}{\operatorname{Nm}}
\newcommand{\R}{\mathbb{R}}

\newcommand{\bk}[1]{\left(#1\right)} %brackets
\newcommand{\bm}{\begin{multline*}}
\newcommand{\tu}{\end{multline*}}

\DeclareMathOperator{\Id}{\mathbbm{1}} %identity element 1
\newcommand{\Gm}{\mathbb{G}_m} %multiplicative group
\newcommand{\modf}[1]{\mathcal{\delta}_{#1}} %modular function of a group
\renewcommand{\check}[1]{#1 ^{\vee}} %for coroots
\newcommand{\piece}[1]{\left\{\begin{matrix} #1 \end{matrix}\right.} %piecewise functions
\newcommand{\set}[1]{\left\{ #1 \right\}} %sets
\newcommand{\mvert}{\mathrel{}\middle\vert\mathrel{}} %midle vert line inside a set
\newcommand{\res}[1]{\Bigg\vert_{#1}}

\newcommand{\Eisen}{\mathcal{E}}
\newcommand{\lmod}{\backslash}
\newcommand{\rmod}{/}
\newcommand{\gen}[1]{\left< #1 \right>}

\newcommand{\intl}{\, \int\limits}
\newcommand{\suml}{\, \sum\limits}
\newcommand{\prodl}{\, \prod\limits}

\theoremstyle{plain}

\usepackage{amsmath,amssymb,verbatim,amsthm,hyperref,bbm, mathrsfs,amscd,pb-diagram}
\usepackage{aliascnt} %for environment names
\usepackage{appendix}

\usepackage{tikz}
\usepackage{verbatim}

\DeclareMathAlphabet{\mathpzc}{OT1}{pzc}{m}{it}

\usepackage[all,cmtip]{xy}

\usepackage{multirow,longtable}
\usepackage{float}

\usepackage[bordercolor=gray!20,backgroundcolor=blue!10,linecolor=none,textsize=footnotesize,textwidth=0.8in]{todonotes}
\usepackage{fancyhdr}
\fancypagestyle{plain}{%
        \fancyhf{} % clear all header and footer fields
        \fancyfoot[C]{\normalfont \thepage} % except the center

}

\title[The  Standard $\Lfun$-function of $G_2$ and the Rallis-Schiffmann lift]{Poles of the  Standard $\Lfun$-function of $G_2$ and the Rallis-Schiffmann lift}
\author[Nadya Gurevich and Avner Segal]{Nadya Gurevich${^{1}}$ and Avner Segal${^{2}}$}
\address{${^1}$ School of Mathematics, Ben Gurion University of the Negev, POB 653, Be'er Sheva 84105, Israel}
\address{${^2}$ School of Mathematical Sciences, Tel Aviv University, Tel Aviv 69978, Israel}
\email{ngur@math.bgu.ac.il, avners@post.tau.ac.il}
\numberwithin{equation}{section}

\allowdisplaybreaks

\begin{document}

\begin{abstract}
We characterize the cuspidal representations of $G_2$ whose standard $\Lfun$-function admits a pole at $s=2$ as the image of Rallis-Schiffmann lift for the commuting pair $\bk{\widetilde{SL_2}, G_2}$ in $\widetilde{Sp_{14}}$. 
The image consists of non-tempered representations.
The main tool is the recent construction, by the second author, of a family of Rankin-Selberg integrals representing  the standard $\Lfun$-function.
%In \cite{MR1020830}, S. Rallis and G. Schiffmann have considered the restriction from $SO_7$ to $G_2$ of the theta-lift $\theta_\psi\bk{\sigma}$ from a representation $\sigma$ of $\TSL$.
%If $\sigma$ satisfy the conditions of the Saito-Kurakawa lift, studied in \cite{MR689647} by I. Piatetski-Shapiro, then  $\theta_\psi\bk{\sigma}$ restricts to a cuspidal representation of $G_2$.
%This representations, called Rallis-Schiffmann lifts, are the \textbf{CAP} representations of $G_2$ associated with a maximal parabolic subgroup of $G_2$.
%
%Moreover, Rallis and Schiffmann proved that if a cuspidal representation $\pi$ of $G_2$ is attained in such a way then its standard partial $\Lfun$-function $\Lfun\bk{\pi,s,\st}$ admits a pole at $s=2$.
%In this paper we prove the converse.
%The lift in the opposite direction was studied in \cite{MR2262172}
%Given an irreducible cuspidal representation $\pi$ of $G_2$ such that $\Lfun\bk{\pi,s,\st}$ admits a pole at $s=2$ then $\pi$ is in the image of the Rallis-Schiffmann lift.
%In particular, this pole is of order at most $2$ and it is of order $2$ if and only if $\pi$ is in the cuspidal image of the theta lift from $S_3$.
\end{abstract}

\maketitle

\begin{center}
\end{center}

\tableofcontents

\section{Introduction}

Let $G$ be a reductive algebraic group defined over a number field $F$.
In the theory of automorphic forms, for any automorphic irreducible representation $\pi=\otimes_\nu \pi_\nu$ of $G(\A)$ and a finite dimensional representation $\rho$ of the dual Langlands group
$^LG$ one can associate a partial $\Lfun$-function $\Lfun^S\bk{s,\pi,\rho}$, where $S$ is a finite set of places of the number field $F$ outside of which $\pi_\nu$ is unramified. 
It is defined by 
\[
\Lfun^S\bk{s,\pi,\rho}=\prodl_{\nu\notin S} \bk{\det\bk{I-\rho\bk{t_{\pi_\nu}}{q_\nu}^{-s}}}^{-1},
\]
where $t_{\pi_\nu}$ is a representative of the Satake conjugacy class associated to $\pi_\nu$ and $q_\nu$ is the order of the residue field of $F_\nu$.
Conjecturally, all such $\Lfun$-functions admit meromorphic continuation.
The poles of the $\Lfun$-functions are of special interest since the images of functorial lifts can often be characterized in terms of these poles.

\subsection{Poles of $\Lfun$-Functions of Non-Tempered Representations}

Let $\pi=\otimes_\nu\pi_\nu$ be an irreducible automorphic representation with unitary central character contained in the discrete $L^2$ spectrum of $G$.

If the representation $\pi$ is tempered then the eigenvalues of $\rho(t_{\pi_\nu})$ have complex norm $1$ and hence the partial $\Lfun$-function is holomorphic for $Re\bk{s}>1$.
The pole at $s=1$ can occur.

Assume that there exists a reductive group $H$ and an automorphic representation 
$\sigma$ of $H(\A)$ such that:
\begin{enumerate}
\item
There is a map $r: {^LH}\rightarrow {^LG}$ such that $\rho\circ r $, as a representation of $^LH$, contains the trivial representation with some positive multiplicity $n$.
In other words $\rho=n\cdot 1\oplus \rho'$ for some representation $\rho'$ of $^LH$.
\item 
Outside a finite set of places it holds that $r(t_{\sigma_\nu})=t_{\pi_\nu}$. 
\end{enumerate}
In this case, $\pi$ is called a {\sl weak lift} from $\sigma$ of $H(\A)$.
Then $\Lfun^S\bk{s,\pi,\rho}=\bk{\zeta^S\bk{s}}^n \Lfun^S\bk{s,\pi,\rho'}$ is expected to have a pole of order $n$ at $s=1$ since ,by the generalized Riemann hypothesis (GRH), the term $\Lfun^S\bk{1,\pi,\rho'}$ is non-zero.
%Then $\Lfun^S\bk{s,\pi,\rho}=\bk{\zeta^S\bk{s}}^n \Lfun^S\bk{s,\pi,\rho'}$ and the $\Lfun$-function  has a pole of order $n$ at $s=1$ since ,by the generalized Riemann hypothesis, the term $\Lfun^S\bk{s,\pi,\rho'}$ is non-zero.
The expectation is that the converse is also true.
It has been verified in many cases such as \cite{MR1835288} and \cite{MR1454260}.

If $\pi$ is non-tempered, the poles of the partial $\Lfun$-function can be expected at real $s$ larger than $1$.
The source of the poles is just the same as in the tempered case.
The poles come from shifted Riemann zeta functions. 
Motivated by Arthur's conjectures, the expectations are as follows.

Assume that there exists a reductive group $H$ and an automorphic tempered representation $\sigma$ of $H(\A)$ such that:
\begin{enumerate}
\item
There is a map $r: {^LH}\times SL_2(\C)\rightarrow {^LG}$ such that $\rho\circ r$ restricted to $^LH$  contains the trivial representation.
We can write $\rho\circ r=\oplus(\rho_j\otimes Sym^j)$ and let $k$ be the maximal index such that $\rho_k$ contains the trivial representation.
Denote its multiplicity in $\rho_k$ by $n$.
 
\item
$\pi$ is a weak lift of the representation $\sigma\boxtimes 1$ of $H(\A)\times SO(2,1)(\A)$
with respect to the map $\rho\circ r$.
\end{enumerate}
Then the partial $\Lfun$-function $\Lfun^S\bk{s,\pi,\rho}$ is expected to have a pole at $s=\frac{k}{2}+1$ of order $n$. 

Indeed,
\[
\Lfun^S\bk{s,\pi,\rho}=\prodl_j \Lfun^S\bk{s,\sigma\boxtimes 1,\rho_j\otimes Sym^j} .
\]
Each factor can be written as a product of shifted partial $\Lfun$-functions of $\sigma$, namely
\[
\Lfun^S\bk{s,\sigma\boxtimes 1, \rho_j\otimes Sym^j}= \prodl_{l=0}^j \Lfun^S\bk{s+j/2-l,\sigma, \rho_j} .
\]

In particular 
\[
\Lfun^S\bk{s,\sigma\boxtimes 1, 1\otimes Sym^k}=\prodl_{l=0}^k \zeta^S\bk{s+\frac{k}{2}-l} .
\]

The partial zeta function $\zeta^S(s)$ has a pole at $s=1$ and, for $\FNorm{S}>1$, a zero of order $\FNorm{S}-1$ at $s=0$.
The term $\Lfun^S\bk{s,\sigma\boxtimes 1, 1\boxtimes Sym^k}$ admits a pole at $s=\frac{k}{2}+1$.
It cannot be canceled neither by other zeta factors, by the maximality of $k$, nor by other terms, by GRH.
%Hence, by the maximality of $k$, the term $\Lfun^S\bk{s,\sigma\otimes 1, 1\otimes Sym^k}$ has a pole at  $s=\frac{k}{2}+1$ that, by the generalized Riemann hypothesis, can not be canceled by other zeta factors.
Thus $\Lfun^S\bk{s,\pi,\rho}$ is expected to have a pole of order $n$ at $s=\frac{k}{2}+1$.

Again, it is expected that the converse is true, namely that the existence of a real pole $s>1$ of $\Lfun^S\bk{s,\pi,\rho}$ will imply the existence of a group $H$ and a representation $\sigma$ as above.
This was also proved in various cases, such as \cite{MR965059}.

The proof of the above expectation requires both a proof of the existence of the weak functorial lift and some information on the poles of the $\Lfun$-function.

The existence of the functorial lift can be proved by providing an explicit construction,
such as theta lift or any of its variations.

Information on the poles of $\Lfun$-function can usually be obtained using a Rankin-Selberg integral representation.
We demonstrate how the expectations come true in the following setting: 
\begin{itemize}
\item $G$ is the exceptional split group of type $G_2$, so that $\bk{^LG}^0=G_2\bk{\C}$.
\item $\rho=st: G_2\bk{\C}\rightarrow GL_7\bk{\C}$ is the standard representation.
\item $H=SO\bk{2,1}$ is the split special orthogonal group of rank $1$, so that $\bk{^LH}^0=SL_2\bk{\C}$.
\item $r:SL_2(\C)\times SL_2(\C)\rightarrow G_2(\C)$ is the map described below in \Cref{functoriality}.
\end{itemize}

The explicit construction of the weak lift above is fully realized using the Rallis-Schiffmann lift
that will be described below.

\subsection{The Rallis-Schiffmann Lift}

A remarkable construction of a class of cuspidal non-tempered representations of the exceptional group $G_2$ has been obtained by Rallis and Schiffmann in \cite{MR1020830}.
The pair $\bk{SL_2,G_2}$ is not a dual pair, but merely a commuting pair inside $Sp_{14}$.
Indeed, the centralizer of a certain embedding of $SL_2$ in $Sp_{14}$ is the group $SO_7\times \{\pm 1\}$ and $G_2$ is naturally embedded in the split special orthogonal group $SO_7$.

Let $\psi$ be a fixed additive complex character of $F\backslash \A$.
For any cuspidal representation $\sigma$ of the metaplectic cover $\TSL$, its theta lift 
$\theta_\psi(\sigma)$ to $SO_7$ is a non-zero and non-cuspidal automorphic representation.
However, the restriction of functions in $\theta_\psi(\sigma)$ to the subgroup $G_2(\A)$ of $SO_7(\A)$ 
defines a square-integrable automorphic representation of $G_2(\A)$.
This representation of $G_2\bk{\A}$ is cuspidal whenever the theta lift of $\sigma$ to $SO\bk{2,1}$ with respect to $\psi$ is zero.
This lift is denoted by $RS_\psi(\sigma)$.
%Such $\sigma$ are precisely the representations of $\TSL\bk{\A}$ such that the theta lift of $\sigma$ to $SO_5$ is cuspidal, these representations are called Saito-Kurokawa lifts and they were studied in \cite{MR689647}.

It is not difficult to extend the definition of the lift from the cuspidal spectrum to all the summands of the discrete spectrum of $\TSL$.
This is done in \cite[Section 12.7]{MR2262172}.
 
The lift $RS_\psi(\sigma)$ is not necessarily irreducible.
The question of reducibility has been studied in \cite{MR2262172} by complete determination of the local lift.
The local lift is used in order to define certain non-tempered A-packets on $G_2$ and the global lift is used to prove Arthur's multiplicity formula for these packets in most cases. 

The construction is proven to be functorial. 
Although the group $\TSL$ is not algebraic, the Satake parameters of an irreducible automorphic $\sigma$
of $\TSL(\A)$ are defined via the Waldspurger map $Wd_{\psi}$ that associates to every square integrable irreducible automorphic representation $\sigma$ of $\TSL\bk{\A}$ an irreducible square integrable representation $\tau=Wd_\psi\bk{\sigma}$ of $SO\bk{2,1}\bk{\A}$.
The map is finite-to-one.

%\begin{center}
%\[
%\xymatrix{
%& & \underset{RS_\psi\bk{\sigma}}{\mathcal{A}_0\bk{G_2}} \\
%\underset{Wd_\psi\bk{\sigma}}{\mathcal{A}^2\bk{SO\bk{2,1}}} & \underset{\sigma}{\mathcal{A}_0\bk{\widetilde{SL_2}}}
%\ar@{->}[l]_(0.4){Wd_\psi}
%\ar@{->}[ru]^{\theta^{14}_\psi\Big\vert_{G_2}} \ar@{->}[r]^{\theta^{10}_\psi} \ar@{->}[rd]_{\theta^{6}_\psi} &
%\underset{\theta^{10}\bk{\sigma}}{\mathcal{A}_0\bk{SO_5}} \\
%& & \underset{\theta^{6}\bk{\sigma}=0}{\mathcal{A}\bk{SO\bk{2,1}}}
%}
%\]
%{\bf F{\scriptsize IGURE}}: The various lifts from $\widetilde{SL_2\bk{\A}}$
%\end{center}

Thus, whenever $\sigma_\nu$ is spherical, the Satake parameter $t_{\psi,\sigma_\nu}$ is defined 
to be the Satake parameter of $Wd_{\psi_\nu}\bk{\sigma_\nu}$.
Note the dependence on the character $\psi$.

Recall that there are four conjugacy classes of homomorphisms $SL_2\bk{\C}\rightarrow G_2\bk{\C}$ corresponding to four nontrivial unipotent conjugacy classes in $G_2\bk{\C}$.

Consider the map
\begin{equation}
\label{functoriality}
r:SL_2\bk{\C}\times SL_2\bk{\C}\rightarrow G_2\bk{\C},
\end{equation}
where the map restricted to the first copy of $SL_2\bk{\C}$  corresponds to the unipotent conjugacy class generated by a long root and the map restricted to the second copy of $SL_2\bk{\C}$ corresponds to the unipotent conjugacy class generated by a short root of $G_2$.

\begin{Prop}
[\cite{MR1020830}]
Let $\sigma=\otimes_\nu \sigma_\nu$ be a cuspidal representation of $\TSL\bk{\A}$.
Let  $\pi=\otimes_\nu \pi_\nu$ be an irreducible summand of $RS_\psi\bk{\sigma}$. 
For all $\nu$ such that $\sigma_\nu,\pi_\nu$ are unramified, denote by $t_{\psi,\sigma_\nu}$ the Satake parameter of $\sigma_\nu$  and by $t_{\pi_\nu}$ the Satake parameter of $\pi_\nu$, which is a semisimple conjugacy class in $G_2\bk{\C}$.
Then  
\begin{equation}
r\bk{t_{\psi,\sigma_\nu}, \diag\bk{q_\nu^{1/2},q_\nu^{-1/2}}}=t_{\pi_\nu}.
\end{equation}
\end{Prop}

%In particular, $\pi$ is a weak functorial lift of the automorphic representation $\tau=Wd_\psi\bk{\sigma}\otimes 1$ of $SO\bk{2,1}\times SO\bk{2,1}$ with respect to the map $r$. 
In particular, $\pi$ is a weak functorial lift of the automorphic representation $Wd_\psi\bk{\sigma}\boxtimes 1$ of $SO\bk{2,1}\times SO\bk{2,1}$ with respect to the map $r$. 

\subsection{Near Equivalence Classes}
Let $\sigma$ be an irreducible automorphic representation of $\TSL(\A)$ contained in the square integrable spectrum.
The near equivalence classes of the representation $RS_\psi\bk{\sigma}$ have been studied in \cite{MR2262172}.
The expectation is that the whole near equivalence class of $RS_\psi\bk{\sigma}$ is contained in the image of the lift $RS_\psi$.
To describe the precise result we first recall the structure of the space $\mA_2\bk{\TSL}$ which is the sum of all irreducible representations contained in the space of square integrable genuine automorphic forms of $\TSL$.
Recall from \cite{MR1103429,MR577010} the decomposition of the space $\mA_2\bk{\TSL}$:
\[
\mA_2\bk{\TSL}=\bk{\oplus_\tau \mA_\tau}\oplus 
\bk{\oplus_\chi \mA_\chi},
\]
where $\tau$ runs over the cuspidal representations of $SO\bk{2,1}$ and $\chi$ runs over the set of quadratic Hecke characters of $F^\times\lmod\A^\times$.

For each cuspidal representation $\tau$ of $SO\bk{2,1}$, the space $\mA_\tau$ is a sum of nearly equivalent representations $\sigma$ such that $Wd_\psi\bk{\sigma}=\tau$. 
All summands appear with multiplicity one in $\mA_\tau$ and form a full near equivalence class.

For each quadratic character $\chi$ the space $\mA_\chi$ is a sum of irreducible summands of the Weil representation $\omega_\chi$ associated to the one dimensional orthogonal space whose discriminant defines the quadratic character $\chi$ via class field theory.
Again, all the summands appear with multiplicity one in $\mA_\chi$ and form a full near equivalence class.

Denote by $V_\chi$ and $V_\tau$ the Rallis-Schiffman lift of $\mA_\chi$ and $\mA_\tau$ respectively.
Any cuspidal irreducible representation $\pi$ in $V_\tau$ is nearly equivalent to a constituent of $\Ind^{G_2}_{P_2} \tau \delta_{P_2}^{1/6}$, where $P_2=M_2U_2$ is the Heisenberg parabolic subgroup of $G_2$ and $\tau$ is regarded, by a pull-back, as
a representation of $M_2\simeq GL_2$.
 
Any cuspidal irreducible $\pi$ in $V_\chi$ is nearly equivalent to a constituent of $\Ind^{G_2}_{P_1} \pi(1,\chi)\delta_{P_1}^{1/5}$, where $P_1=M_1\cdot U_1$ is the non-Heisenberg maximal parabolic subgroup of $G_2$ and $\pi\bk{1,\chi}$ is the unitary principal series representation of $M_1\simeq GL_2$.

Conversely, the following theorem implies that if an irreducible representation $\pi$ is nearly equivalent to a summand of $V_\tau$ or $V_\chi$ then it is isomorphic to such a summand.

\begin{Thm}[\cite{MR2262172}, Theorem $16.1$]
\label{near:eq:GG}
Let $\pi$ be an irreducible cuspidal representation of $G_2\bk{\A}$.
\begin{enumerate}
\item
If $\pi$ is nearly equivalent to a summand of $V_\chi$ for some quadratic Hecke character $\chi$ then $\pi$ is contained in $V_\chi$.

\item If $\pi$ is nearly equivalent to a summand of $V_\tau$ for some cuspidal representation $\tau$ of $SO\bk{2,1}$ then $\pi$ is not orthogonal to $V_\tau$.
\end{enumerate}
\end{Thm}

The  lift in the opposite direction is naturally defined.
For a cuspidal representation $\pi$ of $G_2\bk{\A}$, its lift $RS_\psi\bk{\pi}$ to a representation of $\TSL\bk{\A}$ is the span of the functions 
\[
RS_\psi\bk{\phi,\varphi}\bk{g}=
\integral{G_2}\theta_\psi\bk{\phi}\bk{g,h}\varphi\bk{h} dh ,
\]
where $\varphi\in \pi$, $\phi$ is a Schwartz function on $\A^7$ and $\theta_\psi\bk{\phi}$ is an automorphic theta function on $\widetilde{Sp_{14}}$ restricted to $\TSL\times G_2$.
Computing the constant term of $RS_\psi\bk{\pi}$ using the Schr\"odinger model it is easy to see that $RS_\psi\bk{\pi}$ is necessarily cuspidal and, in particular, is contained in $\mA_2(\TSL)$.

\begin{Cor}
Let $\pi$ be an irreducible cuspidal representation of $G_2\bk{\A}$. 
The following statements are equivalent:
\begin{enumerate}
\item  $RS_\psi\bk{\pi}\neq 0$.
\item $\pi$ is nearly equivalent to a summand of $V_\tau$ for some cuspidal $\tau$ or to a summand of $V_\chi$ for some quadratic $\chi$.
\end{enumerate}
\end{Cor}

The goal of this paper is to describe the set of representations satisfying the conditions in the last Corollary in terms of poles of the standard $\Lfun$-function.

\subsection{The Standard $\Lfun$-Function of $G_2$}

Let $\pi$ be an irreducible cuspidal representation 
of $G_2\bk{\A}$. The partial standard $\Lfun$-function 
$\Lfun^S\bk{s,\pi,st}$ is associated to the seven-dimensional representation $st$ of $G_2\bk{\C}$, the dual group of $G_2$.

\begin{Lem} Let $\pi$ be an irreducible representation of $G_2\bk{\A}$, nearly equivalent to an irreducible summand of $V_\tau$ or $V_\chi$.
Then:
\begin{enumerate}
\item
$\Lfun^S\bk{s,\pi,st}$ has a pole at $s=2$.
\item
The pole is simple unless $\pi\subset V_{\chi_0}$, where $\chi_0$ is the trivial character; in which case the pole is of order $2$. 
\end{enumerate}
\end{Lem}

\begin{proof}
For any $\pi$ nearly equivalent to a representation in $V_\tau$, one has
\[
\Lfun^S\bk{s,\pi,st} = \zeta^S\bk{s-1} \Lfun^S\bk{s-1/2,\tau} \zeta^S\bk{s}
\Lfun^S\bk{s+1/2,\tau} \zeta^S\bk{s+1} .
\]
The factor $\zeta^S\bk{s-1}$ contributes a simple pole at $s=2$.

Similarly, for $\pi$ nearly equivalent to a representation in $V_\chi$, one has 
\[
\Lfun^S\bk{s,\pi,st}= \zeta^S\bk{s-1} \Lfun^S\bk{s-1,\chi} \Lfun^S\bk{s,\chi}^2 \zeta^S\bk{s}
\Lfun^S\bk{s+1,\chi} \zeta^S\bk{s+1}
\]
and hence has a simple pole at $s=2$ for $\chi\neq \chi_0$
and a pole of order $2$ for $\chi=\chi_0$.
\end{proof}

In this paper, we will show that the converse is true. 
Namely, the existence of a pole of the standard partial $\Lfun$-function at $s=2$ characterizes the representations $\pi$ such that $RS_\psi\bk{\pi}\neq 0$. 

The meromorphic continuation of $\Lfun^S\bk{s,\pi,st}$ has been proved in the second author's thesis, \cite{SegalRSGeneral,SegalPhdThesis}, by constructing a family of new-way Rankin-Selberg integrals for $\Lfun^S\bk{s,\pi,st}$.
The integrals in the family are parameterized by \'etale cubic algebras $E$ over $F$.
Precisely, for any \'etale cubic algebra $E$ there is an associated simply-connected quasi-split group $H_E$ of type $D_4$ with an Heisenberg maximal parabolic subgroup 
$P$.
Let $\E^\ast_P\bk{f,s,h}$ denote the normalized Eisenstein series associated to the normalized induced representation $\Ind^{H_E}_{P} \modf{P}^s$.
Consider a  family of integrals
\begin{equation}
\zint_E\bk{\varphi,f,s} = \integral{G_2} \varphi\bk{g} \E^\ast_P\bk{f,s,g}dg,
\end{equation}
where $\varphi$ belongs to the space of a cuspidal representation $\pi$ of $G_2\bk{\A}$.

For each cuspidal $\pi$ the integral $\zint_E\bk{\cdot,\cdot,s}$ either represents the standard $\Lfun$-function or is identically zero, depending on whether $\pi$ supports the Fourier coefficient corresponding to the \'etale 
algebra $E$ along the Heisenberg unipotent subgroup or not.
See \Cref{sec:FC} for details on the Fourier coefficients.
It is shown in \cite[Theorem 3.1]{MR2181091}, that any cuspidal representation supports such a coefficient for at least one \'etale cubic algebra $E$.
For such $E$ one has
\begin{equation}
\label{eq:ZintRepresentsLfun}
\zint_E\bk{\varphi,f,s} = \Lfun^S\bk{5s+1/2,\pi,st} d_S\bk{f,\varphi,s}
\end{equation}
and for any given point the data $\varphi$ and $f$ can be chosen so that $d_S\bk{f,\varphi,s}$ is holomorphic and non-zero in a neighborhood of this points.
In this case we say that $\zint_E\bk{\varphi,f,s}$ is an \emph{integral representation} of $\Lfun^S\bk{5s+1/2,\pi,st}$.

If $\zint_E\bk{\varphi,f,s}$ represents $\Lfun^S\bk{5s+1/2,\pi,st}$ then $\zint_E\bk{\varphi,f,s}$ can be used to study the special values of $\Lfun^S\bk{s,\pi,st}$.
In particular, the order of $\Lfun^S\bk{s,\pi,st}$ at $s_0$ is bounded by the order of $\E^\ast_P\bk{f,s,g}$ at $\frac{1}{5}\bk{s_0-\frac{1}{2}}$.
In the right half-plane, the poles of $\E^\ast_P\bk{f,s,g}$ coincide with the poles of the unnormalized Eisenstein series $\E_P\bk{f,s,g}$.
In particular, the order of $\Lfun^S\bk{s,\pi,st}$ at $s_0=2$ is bounded by the order of $\E_P\bk{f,s,g}$ at $s=3/10$.

The poles of $\E_P\bk{f,s,g}$  at $s=3/10$ have been studied at \cite{MR1918673} for $E=F\times F\times F$ and in the second author's thesis \cite{SegalPhdThesis} for general $E$.

\begin{Thm}[\cite{SegalPhdThesis}, \cite{MR1918673}] \label{poles:Eisen_P}
Let $\E_P\bk{f,s,g}$ be an Eisenstein series associated to the representation $I_P\bk{s}$ of the group $H_E$.
\begin{enumerate}
\item
Let $E$ be a cubic field extension.
Then $\E_P\bk{f,s,g}$ is holomorphic at $s=3/10$.

\item
Let $E=F\times K$, where $K$ is a quadratic field extension.
Then $\E_P\bk{f,s,g}$ has at most a simple pole at $s=3/10$.
This pole is attained by the spherical section.
The residual representation at this point is not square integrable.
  
\item
Let $E=F\times F\times F$.
Then $\E_P\bk{f,s,g}$ has a pole of order at most $2$ at $s=3/10$.
This pole is attained by the spherical section. 
The leading term of the Laurent expansion generates the minimal representations of the group $H_E$. 
\end{enumerate}
\end{Thm}

The integrals $\zint_E(\cdot,\cdot,s)$ can also be used to characterize functorial lifts in terms of special values of the standard $\Lfun$-function. 

Let us give an example.
The dual pair $S_3\times G_2\hookrightarrow Spin_8\rtimes S_3$ has been studied locally in \cite{MR1637485} and globally in \cite{MR1932327}.
The minimal representation $\Pi_{min}$ of $Spin_8\rtimes S_3$ can be realized automorphically as
\[
\coset{\bk{s-3/10}^2\E_{P}\bk{\cdot,s,g}}\res{s=3/10} .
\]
Let us denote by $\Theta$ the theta correspondence for this dual pair.

An immediate corollary of the main theorem in \cite{MR3284482} is:
\begin{Thm}
Let $\pi$ be a cuspidal representation of $G_2\bk{\A}$. The following statements are equivalent:
\begin{enumerate}
\item
The partial $\Lfun$-function $\Lfun^S\bk{s,\pi,st}$ admits a double pole at $s=2$.
\item
The theta lift $\Theta\bk{\pi}$ is not zero. 
\end{enumerate}
\end{Thm}
 
Now we can formulate our main theorem which describes 
another result of this type. 
\begin{Thm}
\label{main}
Let $\pi$ be an irreducible cuspidal representation of $G_2\bk{\A}$.
\begin{enumerate}
\item
The standard partial $\Lfun$-function $\Lfun^S\bk{s,\pi,st}$ has a pole at $s=2$ if and only if $RS_\psi\bk{\pi}\neq 0$.

\item
The pole is of order at most $2$.
It is of order $2$ if and only if the representation $\pi$ is contained in $V_{\chi_0}$ where $\chi_0$ is the trivial character.
\end{enumerate}
\end{Thm}

\begin{proof}
The second part is immediate.
If the pole at $s=2$ is double then $\Theta\bk{\pi}\neq 0$.
It follows from \cite[Theorem 13.1]{MR2262172} that such $\pi$  is contained in $V_{\chi_0}$.

The first part will be proved in the last section.
Its proof uses a see-saw duality and a Siegel-Weil type identity for the leading terms of Eisenstein series on a quasi-split group of type $D_4$.
The identities of this type have appeared before in \cite{MR1289491}, \cite{MR1174424}
and \cite{MR1761622}.
\end{proof}

{\bf Acknowledgments.} 
Parts of this project were performed during the authors visit in the 2015 "Representation Theory and Number Theory" workshop in the National University of Singapore and we wish to thank them for their hospitality.
We would like to thank Chengbo Zhu and Hung Yean Loke for helping us with a few questions regarding the structure of the induced representations at Archimedean places.

We are happy to thank Wee Teck Gan for very helpful discussions.

The authors were partially supported by grants 1691/10 and 259/14 from the Israel Science Foundation.

\section{Wave Front} \label{sec:FC}

Relating the partial $\Lfun$-function of the cuspidal representation $\pi$ to its Rankin-Selberg integral requires information on the set of non-degenerate Fourier coefficients along a certain unipotent subgroup that the representation supports.
We shall define the relevant Fourier coefficient below.
We shall also define the Fourier-Jacobi coefficients and describe the connection between them.

\subsection{Fourier Coefficients of $\TSL$}

Let $B=T\cdot N$ denote the Borel subgroup of $SL_2$.
We denote by $\alpha$ the unique positive root of $SL_2$ and denote by $x_\alpha:\G_a\rightarrow N$ the associated one-parametric subgroup.
The torus $T\bk{F}$ acts on the set of non-trivial characters of $N\bk{\A}$ that are trivial on $N\bk{F}$ and the orbits are parameterized by quadratic \'etale algebras.
Fix a non-trivial unitary character $\psi:F\lmod \A\rightarrow \C^\times$.
For any square class $a$ and its associated quadratic algebra $K$ define the character $\Psi_K:N\bk{\A}\rightarrow \C$, given by $\Psi_K\bk{x_\alpha\bk{r}}=\psi\bk{ar}$.
%For any  $K=F\bk{\sqrt{a}}$ define the character $\Psi_K:N\bk{\A}\rightarrow \C$, given by $\Psi_K\bk{x_\alpha\bk{r}}=\psi\bk{ar}$.
 
Let $\TSL\bk{\A}$ denote the metaplectic cover of $SL_2\bk{\A}$.
The groups $N\bk{\A}$ and $T\bk{F}$ split in $\TSL\bk{\A}$.
 
For any automorphic form $\varphi$ of $\TSL$ define 
\[
\varphi^{N,\Psi_K}\bk{g} = \integral{N} \varphi\bk{ng} \overline{\Psi_K\bk{n}}dn .
\]

The wave front of an automorphic representation $\sigma$ of $\TSL\bk{\A}$ is defined by 
\[
\widehat{F}_\psi\bk{\sigma} = 
\set{K \mvert \exists \,\varphi \in \sigma: 
\quad \varphi^{N,\Psi_K}\neq 0}.
\]

%The following proposition is proved in \cite{MR693355}.
As explained in \cite[Subsection 12.3]{MR2262172} we have:
\begin{Prop}
Let $\sigma$ be an irreducible summand of $\mA_2\bk{\TSL}$.
If $\widehat{F}_\psi\bk{\sigma}=\set{K}$ then $\sigma$ is a summand of $\mA_\chi$, where the quadratic character $\chi$ is associated to $K$ by class field theory. 
\end{Prop}    

\subsection{Fourier Coefficients of $G=G_2$}
We shall start with an overview of $G=G_2$ as a Chevalley group, defined over $\Z$.
We fix a maximal split torus $T_G$ and a Borel subgroup $B_G=T_G\cdot N_G$.
This determines the root datum of the group. 
Denote by $\alpha$ and $\beta$ the short and long simple roots respectively.
There are six positive roots 
\[
\Phi^+=\set{\alpha, \beta, \alpha+\beta, 2\alpha+\beta, 3\alpha+\beta, 3\alpha+2\beta} .
\]
For any root $\gamma$ we denote the associated one parametric subgroup by $x_\gamma:\G_a\rightarrow N_G$ and denote its image by $U_\gamma$.
 
We denote by $P_1=M_1\cdot U_1$ and $P_2=M_2\cdot U_2$ 
the maximal standard parabolic subgroups such that
$U_{\alpha}\subset U_1$ and $U_\beta\subset U_2$.

The Levi factor $M_2\bk{F}$ acts on the set of unitary characters on $U_2\bk{\A}$ trivial on $U_2\bk{F}$ and the orbits are indexed by cubic algebras over $F$.
The generic orbits correspond to \'etale cubic algebras. For any cubic algebra $E$ choose a representative $\Psi_E$ of the associated orbit.

For any automorphic form $\varphi$ on $G_2$ and an \'etale cubic algebra $E$ denote 
\[
\varphi^{U_2,\Psi_E}\bk{g} = \integral{U_2}\varphi\bk{ug} \overline{\Psi_E\bk{u}}\,du .
\]
For any automorphic representation $\pi$ of $G_2\bk{\A}$ 
define the wave front of $\pi$ with respect to $U_2$ by
\[
\widehat{F}_\psi\bk{\pi} = 
\set{E \mvert \exists \,\varphi \in \pi: \quad \varphi^{U_2,\Psi_E}\neq 0} .
\]
We shall write down explicitly a character $\Psi_{F\times K}$ on $U_2\bk{\A}$ that is a representative of the generic orbit corresponding
to the \'etale cubic algebra $F\times K$, where $K$ is a quadratic \'etale algebra.
Let $a$ be the square class in $F^\times$ associated 
to $K$.
Then 
\[
\Psi_{F\times K}\bk{x_\beta\bk{r_1}
x_{\alpha+\beta}\bk{r_2} x_{2\alpha+\beta}\bk{r_3} x_{3\alpha+\beta}\bk{r_4} x_{3\alpha+2\beta}\bk{r_5}} =
\psi\bk{-ar_1+r_3} .
\]

A family of reductive periods is closely related to the family of Fourier coefficients corresponding to the algebras of type $F\times K$.
The group $G_2$ acts on the seven-dimensional space preserving the split quadratic form.  
The stabilizer of a vector in this space, whose norm is a square class in $F^\times$ corresponding to the algebra $K$, is isomorphic to the special unitary group $SU^K_3$. 
In particular, when $K=F\times F$ the stabilizer is isomorphic to $SL_3$. 

For an irreducible cuspidal representation $\pi$ of $G_2\bk{\A}$ define 
\[
\widetilde{F}\bk{\pi}=
\set{ K \mvert \exists \, 
\varphi \in \pi: \quad \integral{SU^K_3}\varphi\bk{g}\, dg\neq 0} .
\]

The following properties of Fourier coefficients will be used.
They are proved in \cite{MR2262172} and \cite{MR1020830}.

\begin{Prop}[\cite{MR1020830,MR2262172}]
\label{FC}
 Let $\pi$ be an irreducible cuspidal representation of 
$G_2\bk{\A}$ and $\sigma$ be an irreducible representation in $\mA_2\bk{\TSL}$.
\begin{enumerate}
\item
If $K\in \widetilde{F}\bk{\pi}$ then $F\times K\in  \widehat{F}_\psi\bk{\pi}$.
\item 
$K\in \widehat{F}_\psi\bk{RS_\psi\bk{\pi}}\Leftrightarrow F\times K\in \widehat{F}_\psi(\pi)$.
\item
If $\sigma$ is contained in $RS_\psi\bk{\pi}$ then 
$\widehat{F}_\psi\bk{\sigma}\subseteq\widetilde{F}\bk{\pi}$.
\item
$RS_\psi\bk{\pi}\neq 0$ if and only if $\widetilde{F}\bk{\pi}\neq \emptyset$.
\end{enumerate} 
\end{Prop}

\subsection{Fourier-Jacobi Coefficients}

Let $P_1=M_1U_1$ be a non-Heisenberg maximal parabolic subgroup of $G_2$.
The unipotent radical $U_1$ is the two-step unipotent 
subgroup.
One has $U_1\supset Z\supset Z_1$ where $Z=[U_1,U_1]$ and $Z_1$ is the center of $U_1$.
The group $U_1\rmod Z_1$ is the Heisenberg group with center $Z\rmod Z_1\simeq \G_m$.
We regard $\psi$ as the character of $Z\rmod Z_1$.
The Weil representation $\omega^2_\psi$ of $U_1\bk{\A}\rmod Z_1\bk{\A}$ is realized on the space of Schwarz functions $\Sch \bk{U_\alpha}$ and is extended to the group
$\widetilde{M'_1}\bk{\A} \cdot \coset{U_1\bk{\A}\rmod Z_1\bk{\A}}$, where $M_1'\simeq SL_2$ is the derived group of $M_1$. 

The representation $\omega^2_\psi$ has an automorphic realization in the space of Jacobi forms by 
\[
\theta_\psi\bk{\phi} = \sum_{x\in U_\alpha\bk{F}} \omega_\psi\bk{g}\phi\bk{x}, \quad 
g\in \TSL\bk{\A} \cdot \coset{U_1\bk{\A}\rmod Z_1\bk{\A}}.
\]

For an automorphic form $\varphi$ of $G_2$, its Fourier-Jacobi coefficient is defined by 
\[
\varphi^{Z,\psi}\bk{g}=\integral{Z} \varphi\bk{zg} \overline{\psi\bk{z}} dz,
\]
which is a Jacobi form on $\TSL\bk{\A} \cdot \coset{U_1\bk{\A}/Z_1\bk{\A}}$.
The space generated by all such coefficients is denoted by $\pi_{Z,\psi}$.
This is a representation of the Jacobi group.

The automorphic representation $FJ_\psi\bk{\pi}$ of $\TSL(\A)$ is defined as the span of all the functions
\[
FJ_\psi\bk{\varphi,\phi}\bk{h} =
\integral{U_1} \varphi^{Z,\psi}\bk{uh} \overline{\theta_\psi\bk{\phi}\bk{uh}}\, du,
 \quad 
h\in \TSL\bk{\A}, \quad \varphi\in \pi, \quad \phi\in \omega^2_\psi
\]
An easy computation shows: 
\begin{Prop}
\label{FJ:FC}
Let $\pi$ be an automorphic representation of $G_2$ and $K$ be a quadratic \'etale algebra over $F$.
Then
\[
K\in \widehat{F}_{\overline\psi}\bk{FJ_\psi\bk{\pi}}\Leftrightarrow  
F\times K\in \widehat{F}_\psi\bk{\pi}.
\]
\end{Prop}

\subsection{The Pole of the $\Lfun$-Function  and Fourier Coefficients}

The next theorem gives information on the Fourier coefficients supported by an irreducible cuspidal representation $\pi$ of $G_2\bk{\A}$, whose standard $\Lfun$-function admits a pole at $s=2$.

\begin{Thm}
\label{wave-front}
Let $\pi$ be an irreducible cuspidal representation of $G_2(\A)$ such that the partial $\Lfun$-function  $\Lfun^S\bk{s,\pi,st}$ admits a pole at $s=2$.
Then:
\begin{enumerate}
\item
The representation $\pi$ is not nearly equivalent to a generic cuspidal representation. 

\item
The wave front $\widehat{F}_\psi\bk{\pi}$ does not contain cubic field extensions.

\item
If $\widehat{F}_\psi\bk{\pi}=\{F\times F\times F\}$ then $RS_\psi\bk{\pi}\neq 0$ and the pole is of order $2$. 

\item
If the pole at $s=2$ is simple then there exists 
a quadratic field extension $K$ such that $F\times K\in \widehat{F}_\psi\bk{\pi}$.
\end{enumerate}
\end{Thm}

\begin{proof}
\begin{enumerate}
\item
D. Ginzburg in \cite{Ginzburg1993} has constructed a Rankin-Selberg integral for the standard $\Lfun$-function of {\sl generic} cuspidal representations of $G_2$.
The construction uses an Eisenstein series on $\TSL$.
In particular, it is shown that in the right half plane $\set{s\in\C\mvert \ Re\bk{s}>0}$ the partial $\Lfun$-function of a generic cuspidal representation can have a pole only at $s=1$.

\item
This follows immediately from \Cref{poles:Eisen_P}.

\item
This is the most involved part, and the proof is essentially contained in the proof of Theorem 16.1 in \cite{MR2262172}.
We repeat it for the convenience of the reader.

{\sl \underline{Step 1:} $FJ_{\psi}(\pi)$ is contained in $\overline {\mathcal A_{\chi_0}}$ where $\chi_0$ is the trivial character.}

If $\widehat{F}_\psi\bk{\pi}=\{F\times F\times F\}$ then by \Cref{FJ:FC} $\widehat{F}_{\overline \psi}\bk{FJ_\psi\bk{\pi}}=\{F\times F\}$, so that $FJ_\psi\bk{\pi}\subset \overline {\mathcal A_{\chi_0}}$.

We recall the exceptional dual pair $\bk{G_2,PGL_3}$ in $E_6$ considered in \cite{MR1455531}.

{\sl \underline{Step 2:} The theta lift of $\pi$ to $PGL_3$ is a non-zero representation, orthogonal to the cuspidal spectrum.}

One has 
\[
\pi_{Z,\psi}\simeq FJ_\psi\bk{\pi}\otimes \omega_{\psi}\subset 
\overline{\mathcal{A}_{\chi_0}}\otimes \omega_{\psi} ,
\]
where the first equality is due to \cite{MR1295945}.
In particular the Shalika functional
\[
\beta\bk{\varphi}= \integral{SL_2} \integral{Z} \varphi\bk{zg} \overline{\Psi\bk{z}}\,dz\,dg
\]
is not identically zero on $\pi$.
By \cite[Theorem 3.7]{MR1835288}, non-vanishing of the Shalika functional implies that the theta lift of $\pi$ to $PGL_3$ is non-zero.
The theta lift $\theta\bk{\pi}$ is non-cuspidal.
Otherwise $\theta\bk{\pi}$ would be generic and $\pi$ would be nearly equivalent to a generic cuspidal subrepresentation of $\theta\bk{\theta\bk{\pi}}$.
This contradicts $(1)$.

{\sl \underline{Step 3:} $RS_\psi\bk{\pi}\neq 0$.}

By the tower of theta correspondences for $G_2$ studied in \cite[Theorem A]{MR1455531}, the theta lift of $\pi$ to the group $PGL_2$, that is a lower step in the tower is non-zero.
By \cite[Theorem 4.1.(5)]{MR1455531} this implies that $\pi$ has non-trivial $SL_3$ period.
In particular, $F\times F \in \widetilde{F}\bk{\pi}$ and hence $RS_\psi\bk{\pi}\neq 0$.

{\sl \underline{Step 4:} The pole is of order $2$.}

This also implies, by \Cref{FC}$(3)$ that $\widehat{F}_\psi\bk{RS_\psi\bk{\pi}}=\set{F\times F}$ and hence $RS_\psi\bk{\pi}\subset \mathcal{A}_{\chi_0}$.
So $\pi$ is isomorphic to a summand of $V_{\chi_0}$.
In particular, the partial $\Lfun$-function has a pole of order $2$ at $s=2$.

In fact, it follows from \Cref{near:eq:GG}. that $\pi$ is contained in $V_{\chi_0}$.

\item
Recall from \cite[Theorem 3.1]{MR2181091} that any cusp form of $G_2\bk{\A}$ supports some non-degenerate coefficient along $U_2$.
Parts $(2)$ and $(3)$ imply the statement.
\end{enumerate}
\end{proof}

\section{Induced Representations and  Eisenstein series}
The proof of the main theorem will involve an identity between two Eisenstein series on the group $H=H_{F\times K}$.
To establish the identity we shall first study the reducibility of the  induced representations involved.

In this section we consider two degenerate principal series representations induced from $P$ and another parabolic subgroup $Q$ and also the degenerate Eisenstein series associated with them.
We prove a Siegel-Weil type identity relating the leading terms of the two series at certain points.
%Before discussing the identity we construct an intertwining operator between these two representations.

\subsection{Notations}
\label{Subsec:Notations}
For the sake of this subsection, we let $F$ denote a global or local field.
Let $K$ be a quadratic \'etale algebra over $F$.
We shall write $H$ for  $H_{F\times K}$, the quasi-split simply connected group over $F$ of type $D_4$.
Let $B_H$ be a Borel subgroup of $H$ with unipotent radical $N_H$, a maximal torus $T_H$ and a maximal split torus $T_S$ in $T_H$.
We number the vertices of the Dynkin diagram $D_4$
by $\alpha_i, i=1\ldots 4$ such that $\alpha_2$ is the middle vertex.
\begin{center}
  \begin{tikzpicture}[scale=.4]
    \draw[xshift=0 cm,thick] (0 cm,0) circle (.3cm);
    \draw (0,0.4) node[anchor=south]  {$\alpha_1$};
    \draw[xshift=1 cm,thick] (1 cm,0) circle (.3cm);
    \draw (2,0.4) node[anchor=south]  {$\alpha_2$};
    \draw[xshift=2 cm,thick] (30: 17 mm) circle (.3cm);
    \draw (3.6,1) node[anchor=south]  {$\alpha_3$};
    \draw[xshift=2 cm,thick] (-30: 17 mm) circle (.3cm);
    \draw (3.6,-1) node[anchor=north]  {$\alpha_4$};    
    \draw[xshift=0.15 cm,thick] (0.15 cm,0) -- +(1.4 cm,0);
    \draw[xshift=2 cm,thick] (30: 3 mm) -- (30: 14 mm);
    \draw[xshift=2 cm,thick] (-30: 3 mm) -- (-30: 14 mm);
  \end{tikzpicture}\\
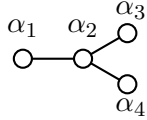
  {\bf F{\scriptsize IGURE}}: The Dynkin diagram of type $D_4$
\end{center}
If $K$ is a field, let $\alpha_1$  be the simple root that is defined over $F$ and the field of definition of $\alpha_3, \alpha_4$ is $K$.
\begin{center}
  \begin{tikzpicture}[scale=.4]
    \draw[xshift=0 cm,thick] (0 cm,0) circle (.3cm);
    \draw (0,0.4) node[anchor=south]  {$\alpha_1$};
    \draw[xshift=1 cm,thick] (1 cm,0) circle (.3cm);
    \draw (2,0.4) node[anchor=south]  {$\alpha_2$};
    \draw[xshift=2 cm,thick] (30: 17 mm) circle (.3cm);
    \draw (3.6,1) node[anchor=south]  {$\alpha_3$};
    \draw[xshift=2 cm,thick] (-30: 17 mm) circle (.3cm);
    \draw (3.6,-1) node[anchor=north]  {$\alpha_4$};    
    \draw[xshift=0.15 cm,thick] (0.15 cm,0) -- +(1.4 cm,0);
   \draw[xshift=2 cm,thick] (30: 3 mm) -- (30: 14 mm);
    \draw[xshift=2 cm,thick] (-30: 3 mm) -- (-30: 14 mm);
    \draw[xshift=3.6 cm,thick]  ellipse (0.8cm and 3cm);
  \end{tikzpicture}\\
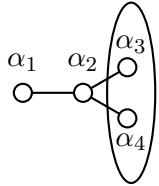
  {\bf F{\scriptsize IGURE}}: The Dynkin diagram of type ${^2}D_4$
\end{center}

The positive roots of $H$ with respect to $T$, in the split case $K=F\times F$, are
\[
\begin{split}
\Phi^{+} = & \left\{
\coset{1,0,0,0}, \coset{0,1,0,0}, \coset{0,0,1,0}, \coset{0,0,0,1}, \coset{1,1,1,1}, \coset{1,2,1,1},
\right. \\
& \left.
\coset{1,1,0,0}, \coset{0,1,1,0}, \coset{0,1,0,1}, \coset{1,1,1,0}, \coset{1,1,0,1}, \coset{0,1,1,1} 
\right\} .
\end{split}
\]
In the case where $K$ is a field, the positive roots of the relative root system of $H$ with respect to $T_S$ are
\[
\Phi^{+} = \left\{
\coset{1,0,0}, \coset{0,1,0}, \coset{0,0,1}, \coset{1,1,0}, \coset{0,1,1}, \coset{0,1,2}, \coset{1,1,1}, \coset{1,2,1}, \coset{1,2,2}
\right\} .
\]
The roots defined over $K$ are $\coset{0,0,1}$, $\coset{0,1,1}$ and $\coset{1,1,1}$.

For a root $\alpha\in\Phi^{+}$ we denote by $F_\alpha$ the field of definition of $\alpha$ and by $\varphi_\alpha:SL_2\bk{F_\alpha}\to H$ the associated embedding.
In particular, we denote by $X_{\alpha}$ and $X_{-\alpha}$ the associated one-parametric subgroups of $H$, namely the images of the maps $x_{\alpha}, x_{-\alpha}:F_\alpha\to H$ given by
\[
\begin{split}
& x_\alpha\bk{r} = \varphi_\alpha\bk{\begin{pmatrix}1&r\\0&1\end{pmatrix}} \\
& x_{-\alpha}\bk{r} = \varphi_\alpha\bk{\begin{pmatrix}1&0\\r&1\end{pmatrix}}
\end{split}
\]
respectively.

Let $W_H=W\bk{H,T_S}$ denote the relative Weyl group of $H$ associated with $T_S$.
For $\alpha\in\Phi^{+}$ let $w_{\alpha}$ denote the simple reflection $\varphi_\alpha\bk{\begin{pmatrix}0&1\\-1&0\end{pmatrix}}$ associated with $\alpha$.
By $w\coset{i_1,...,i_k}$ we denote the Weyl word $w_{\alpha_{i_1}}\cdot\dots\cdot w_{\alpha_{i_k}}$.

Let $P=M\cdot U$ be the Heisenberg group of $H$, obtained by removing 
the root $\alpha_2$. Its Levi subgroup is 
\[
M\simeq \bk{GL_2\times \Res_{K/F}GL_2}^0=
\{(g_1,g_2)\in GL_2\times \Res_{K/F}GL_2: \det(g_1)=\deg(g_2)\}
\]
and is generated by the roots $\alpha_1,\alpha_3,\alpha_4$.
Under this isomorphism it holds that
\[
\delta_{P}(g_1,g_2)=\FNorm{\det(g_1)}^5 \quad \forall g_1\in GL_2, \ g_2\in\Res_{K/F}GL_2 .
\]

Let $Q=L\cdot V$ be the maximal parabolic subgroup obtained by 
removing the root $\alpha_1$. Its Levi subgroup is 
\[
L\simeq GL_1 \ltimes Spin_6^K,
\]
where $Spin_6^K$ is the quasi-split form of $Spin_6$ generated by the roots $\alpha_2$, $\alpha_3$ and $\alpha_4$.
Under this isomorphism it holds that
\[
\delta_{Q}(g_1,g_2)=\FNorm{g_1}^{6} \quad \forall g_1\in GL_1,\ g_2\in Spin_6^K .
\]

In the split case we parameterize the torus $T\simeq \G_m^4$ using the coroots,
\[
(t_1,t_2,t_3,t_4)=
\alpha_1^\vee(t_1)
\alpha_2^\vee(t_2)
\alpha_3^\vee(t_3)
\alpha_4^\vee(t_4),\quad t_1,t_2,t_3,t_4\in\Gm .
\]
Then
\[
\modf{B_H}\bk{t_1,t_2,t_3,t_4} = \FNorm{t_1t_2t_3t_4}_F^2 ,
\]
namely $\modf{B}=\lambda_{\bk{-2,-2,-2,-2}}$.

In this quasi-split case, the torus is parameterized by the coroots
\[
(t_1,t_2,t_3)=\alpha_1^\vee(t_1)\alpha_2^\vee(t_2)\alpha_3^\vee(t_3), 
\quad t_1,t_2\in\Gm,\ t_3\in \Res_{K\rmod F}\Gm .
\]

It holds that
\[
\modf{B_H}\bk{t_1,t_2,t_3} = \FNorm{t_1t_2\Nm\bk{t_3}}_F^2 = \FNorm{t_1t_2}_F^2 \FNorm{t_3}_K^2 .
\]

We denote the space of unramified characters by $\mathfrak{a}_\C^\ast = X^\ast_{nr}\bk{T,\C}$ and denote by $C^{+}$ the positive Weyl chamber in $\mathfrak{a}_\C^\ast$ associated to the choice of positive roots above.

The co-ordinates of $\mathfrak{a}_\C^\ast$ are given with respect to the fundamental weights.
Namely, in the quasi-split case $\lambda_{\overline{s}}$ for $\overline{s}=\bk{s_1,s_2,s_3}$ denotes the character given by
\[
\lambda_{\overline{s}}\bk{t_1,t_2,t_3} = \FNorm{t_1}_{F}^{s_1} \FNorm{t_2}_{F}^{s_2} \FNorm{t_3}_{K}^{s_3} .
\]
Similarly, in the split case, $\lambda_{\overline{s}}$ for $\overline{s}=\bk{s_1,s_2,s_3,s_4}$ denotes the character given by
\[
\lambda_{\overline{s}}\bk{t_1,t_2,t_3,t_4} = \FNorm{t_1}_{F}^{s_1} \FNorm{t_2}_{F}^{s_2} \FNorm{t_3}_{F}^{s_3} \FNorm{t_4}_{F}^{s_4}  .
\]

For $\lambda\in\mathfrak{a}_\C^\ast$ we denote $I_{B_H}\bk{\lambda} = \Ind_{B_H\bk{\A}}^{H\bk{\A}}\bk{\lambda} = \otimes_\nu I_{B_{H,\nu}}\bk{\lambda}$, where 
\[
I_{B_{H,\nu}}\bk{\lambda} = \Ind_{B_H\bk{F_\nu}}^{H\bk{F_\nu}}\bk{\lambda},
\]
 the normalized induction.

Consider the induced representations $I_{Q}\bk{s}=\Ind^{H\bk{\A}}_{Q\bk{\A}}\delta_{Q}^{s}=\otimes_\nu I_{Q_\nu}\bk{s}$, where $I_{Q_\nu}\bk{s} = \Ind_{Q\bk{F_\nu}}^{H\bk{F_\nu}}\delta_Q^s$.
We denote the normalized spherical section of $I_Q\bk{s}$ by $\widetilde{f}_s^0=\otimes_\nu \widetilde{f}^0_{\nu,s}$.
Similarly, we define $I_P\bk{s}$ and $I_{P_\nu}\bk{s}$.
We denote the normalized spherical section of $I_P\bk{s}$ by $f^0_s=\otimes_\nu f^0_{\nu,s}$.
We recall that
\[
\begin{split}
& I_Q\bk{s} \hookrightarrow I_{B_H}\bk{\chi_s^Q} \\
& I_P\bk{s} \hookrightarrow I_{B_H}\bk{\chi_s^P},
\end{split}
\]
where
\[
\begin{split}
& \chi_s^Q = \modf{Q}^{s+1/2} \modf{B_H}^{-1/2} =
\piece{\lambda_{\bk{6s+2,-1,-1,-1}},& K=F\times F\\ \lambda_{\bk{6s+2,-1,-1}},& \text{otherwise}}\\
& \chi_s^P = \modf{P}^{s+1/2} \modf{B_H}^{-1/2} = 
\piece{\lambda_{\bk{-1,5s+3/2,-1,-1}},& K=F\times F\\ \lambda_{\bk{-1,5s+3/2,-1,-1}},& \text{otherwise}} .
\end{split}
\]

For $w\in W_H$ we consider the standard global intertwining operator
\[
M\bk{w,\lambda}: I_{B_H}\bk{\lambda} \to I_{B_H}\bk{w^{-1}\cdot\lambda}
\]
by
\[
M\bk{w,\lambda} \bk{f_\lambda}\bk{g} = 
\intl_{N_H\bk{\A}\cap w^{-1}N_H\bk{\A}w\lmod N_H\bk{\A}} f_\lambda\bk{wng} \ dn ,
\]
where $f_\lambda$ is a flat section of $I_B\bk{\lambda}$.
This integral converges absolutely for $\lambda\in C^{+}$ and extends to a meromorphic function on $\mathfrak{a}_\C^\ast$.

\begin{Remark}
\begin{enumerate}
\item
For any $w,w'\in W$, the intertwining operators satisfy the following cocycle equation
\[
M_{ww'}\bk{\chi_s} = M_{w'}\bk{w^{-1}\cdot \chi_s}\circ M_w\bk{\chi_s} .
\]

\item For the shortest representative $w$ of a coset in $W_L\lmod W_H$, $M\bk{w,\chi_s^Q}$ restricts to $I_Q\bk{s}$.
Similarly for $I_P\bk{s}$ and $\chi_s^P$.

\item The global intertwining operator $M\bk{w,\lambda}$ decomposes as a product $\otimes M_\nu\bk{w,\lambda}$ of local intertwining operators $M_\nu\bk{w,\lambda}:I_{B_{H,\nu}}\bk{\lambda} \to I_{B_{H,\nu}}\bk{w^{-1}\cdot\lambda}$ given, for $Re\bk{\lambda}\gg 0$, by
\[
M_\nu\bk{w,\lambda} \bk{f_\lambda}\bk{g} = 
\intl_{N_H\bk{F_\nu}\cap w^{-1}N_H\bk{F_\nu}w\lmod N_H\bk{F_\nu}} f_\lambda\bk{wng} \ dn .
\]
\end{enumerate}
\end{Remark}

\subsection{Induced Representations}

In this subsection we study the reducibility of the local degenerate principal series $I_{Q_\nu}(1/6)$ and $I_{P_\nu}\bk{3/10}$ in various cases.
We fix a place $\nu$ of $F$ and drop $\nu$ from all notations, namely: $F$ is a local field, $K$ is a quadratic \'etale algebra over $F$, $H=H\bk{F}$ and so forth.

\begin{Thm}
\label{I:QP:structure}
\begin{enumerate}
\item
The representation $I_P\bk{3/10}$ has unique irreducible quotient which is spherical.
This quotient is isomorphic to the minimal representation whenever $K$ is not a field.

\item
Assume $F$ is non-Archimedean.
If $K$ is a field then $I_{Q}\bk{1/6}$ is irreducible.
If $K=F\times F$ then $I_{Q}\bk{1/6}$ has length two.
The unique irreducible quotient in this case is isomorphic to the minimal representation.

\item
The local intertwining operator $M\bk{w[2342]}$ defines a map 
\[
M\bk{w[2342],\chi_{3/10}^P}:I_{P}\bk{3/10}\rightarrow I_{Q}\bk{1/6}
\]
whose image contains the submodule $I_{Q}^0\bk{1/6}$ generated by the spherical vector.
\end{enumerate}
\end{Thm}

\begin{Remark}
We treat $I_Q\bk{1/6}$ in the Archimedean case in the end of this section.
\end{Remark}

\begin{proof}
	\begin{enumerate}
	\item
	If $K=F\times F$, let $\overline{s}_0=\bk{1,0,1,1}$, otherwise let $\overline{s}_0=\bk{1,0,1}$.
	We note that
	\[
	I_{B_H}^H \bk{\lambda_{\overline{s}_0}} \twoheadrightarrow I_P\bk{3/10} .
	\]	
	On the other hand, by induction in stages
	\[
	\Ind_{B_H}^{H} \bk{\lambda_{\overline{s}_0}} = \Ind_R^H \coset{\modf{R}^{1/3} \Ind_{X\cap B_H}^X \Id}
	\]	
	where $R$ is the standard parabolic subgroup of $H$ generated by the root $\alpha_2$ and $X$ is the Levi subgroup of $R$.
	Since $\Ind_{X\cap B}^X \Id$ is tempered, by the Langlands classification theorem, $I_{B_H}^{H} \bk{\lambda_{\overline{s}_0}}$ admits a unique irreducible quotient $\Pi^0$.
	Furthermore, $\Pi^0$ is the image of the Langlands operator $M\bk{w[12321421324],\lambda_{\overline{s}_0}}$, where $w[12321421324]$ is the shortest representative of the longest coset $W_X\lmod W_H$.
	Since $M\bk{w[12321421324],\lambda_{\overline{s}_0}}f^0_{\lambda_{\overline{s}_0}}\neq 0$ it follows that $\Pi^0$ is a spherical representation.
	
	Since $I_{P}\bk{3/10}$ is a quotient of $I_{B_H} \bk{\lambda_{\overline{s}_0}}$ it also admits a unique (the same) irreducible quotient $\Pi^0$.
	The claim follows.

	\item
	The reducibility in the split case, where $K=F\times F$, was proven by M. Weissman in his thesis \cite[Subsection 5.1]{MR1993361} using the Fourier-Jacobi functor.
	We now adapt his approach to the quasi-split case.
	The case where $K$ is a field will be dealt with in \Cref{Subsec:LocalFourierJacobi}.

	\item
	Let $R=P\cap Q$.
	Then $M\cap R$ and $L\cap R$ are maximal parabolic subgroups of $M$ and $L$ respectively with a Levi factor is isomorphic to $\bk{F^\times\times F^\times\times GL_2\bk{K}}^0$.
	The group  of unramified characters $X^\ast(M\cap L,\C)$ is isomorphic to $\C^2$.

	We fix the following characters in $X^\ast_{nr}\bk{M\cap L,\C}$:
	\[
	\iota_{P,Q}\bk{s_1,s_2}=\bk{\delta^{M}_{M\cap R}}^{s_1}\bk{\delta_{P}}^{s_2}, 
	\quad 
	\iota_{Q,P}\bk{s_1,s_2}=\bk{\delta^{L}_{L\cap R}})^{s_1}\bk{\delta_{Q}}^{s_2} .
	\]
	In particular, it holds that
	\[
	\Ind^{H}_R\bk{\iota_{P,Q}\bk{s_1,s_2}}=\Ind^{H}_{P}\bk{\Ind^{M}_{M\cap R}\bk{s_1},s_2} \quad 
	\Ind^{H}_R\bk{\iota_{Q,P}\bk{s_1,s_2}}=\Ind^{H}_{Q}\bk{\Ind^{L}_{L\cap R}\bk{s_1},s_2} .
	\]
	The following lemma is obtained by direct computation.
	
	\begin{Lem}
		\label{Lem:SuitablePair}
		\[
		\iota_{P,Q}\bk{s_1,s_2}=\iota_{Q,P}\bk{\frac{5s_2-s_1}{4},\frac{s_1+5s_2}{6}}.
		\]
		In particular,
		\[
		\iota_{P,Q}\bk{-1/2,3/10}=\iota_{Q,P}\bk{1/2,1/6}.
		\]
		\end{Lem}
	
	On the other hand, we note that
	\[
	w_1\cdot \iota_{P,Q}\bk{1/2,3/10} = \iota_{P,Q}\bk{-1/2,3/10} .
	\]

	\begin{Prop}
		\begin{enumerate}
			\item
			The operators 
			\[
			\begin{split}
			& M\bk{w[1],s}: I^{M}_{M\cap R}\bk{s}\rightarrow  I^{M}_{M\cap R}\bk{-s}, \\
			& M\bk{w[2342],s}: I^{L}_{L\cap R}\bk{s}\rightarrow  I^{L}_{L\cap R}\bk{-s} .
			\end{split}
			\]
			define equivariant maps whose images are the trivial representations of $M$ and $L$ respectively.
			
			\item
			The operators above induce operators
			\[
			\begin{split}
			& M\bk{w[1],s}:\Ind^{H}_R\bk{i_{P,Q}\bk{1/2,s}}\rightarrow \Ind^{H}_R\bk{i_{P,Q}(-1/2,s)} \\
			& M\bk{w[2342],s}:\Ind^{H}_R\bk{i_{Q,P}\bk{1/2,s}} \rightarrow \Ind^{H}_R\bk{i_{Q,P}\bk{-1/2,s}}
			\end{split}
			\]			
			whose images are $I_{P}\bk{s}$ and $I_{Q}\bk{s}$ respectively. 
			
			\item 
			The restriction of $M\bk{w[2342],1/6}$ to $I_P\bk{3/10}$ 
			defines a surjective map 
			\[
			M\bk{w[2342],1/6}: I_{P}\bk{3/10}\rightarrow I_{Q}^0\bk{1/6},
			\]
			where $I_{Q}^0\bk{1/6}$ is the subrepresentation of $I_{Q}\bk{1/6}$ generated by the spherical vector.
		\end{enumerate}
	\end{Prop}

	\begin{proof}
		The elements $w[1]$ and $w[2342]$ are the shortest representatives of the longest elements in  $W_{M\cap R}\backslash W_M$ and  
		$W_{L\cap R}\backslash W_L$.
		The intertwining operator is holomorphic at this point due to \cite{MR517138} and \cite{MR563369}.
		The representations $\Ind^{M}_{M\cap R}(1/2)$, $\Ind^{L}_{L\cap R}(1/2)$ are standard modules and hence the images of the intertwining operators above are their Langlands quotients, that is the trivial representations. 
		
		The second part follows from the first by induction in stages.
		
		Since by \Cref{Lem:SuitablePair} $\iota_{P,Q}\bk{-1/2,3/10}=\iota_{Q,P}\bk{1/2,1/6}$ we can restrict $M\bk{w[2342],1/6}$ to the image of $M\bk{w[1],3/10}$ and obtain a map from $I_{P}\bk{3/10}$ to $I_{Q}\bk{1/6}$.
		The image of the spherical vector in $I_{P}\bk{3/10}$ is a non-zero spherical vector in $I_{Q}\bk{1/6}$. 
		The surjectivity follows from the second part of \Cref{I:QP:structure}
	\end{proof}
	This also ends the proof of \Cref{I:QP:structure}.
	\end{enumerate}
\end{proof}

\begin{Remark}
	The above discussion can be summarized by the following diagram
	\[
	\xymatrix{
		\Ind_R^{H}\bk{\iota_{P,Q}\bk{1/2,3/10}} \ar@{->}[r]^{M\bk{w[1],3/10}} \ar@{->>}[d] & \Ind_R^{H}\bk{\iota_{P,Q}\bk{-1/2,3/10}} \ar@{=}[r] & \Ind_R^{H}\bk{\iota_{Q,P}\bk{1/2,1/6}} \ar@{->>}[d]^{M\bk{w[2342],1/6}} \\
		I_{P}\bk{3/10} \ar@{^{(}->}[ur] \ar@{-->}[rr]^{M\bk{w[2342],1/6}} & &
		I_{Q}^0\bk{1/6} \subseteq I_{Q}\bk{1/6} \hspace{1.2cm} \\
	} .
	\]
	
\end{Remark}

\subsection{Irreducibility of $I_Q\bk{1/6}$ when $K$ is a Field}
\label{Subsec:LocalFourierJacobi}
Let $K$ be a field.
In this subsection we prove that $I_{Q}\bk{1/6}$ is irreducible.
First, let us recall some notations.
	
Let $M'=\coset{M,M}$ be the derived group of the Levi $M$ of $P$.
%	We note that
%	\[
%	M'= Im \bk{\varphi_{\alpha_1}} \times_{\mu_2} Im \bk{\varphi_{\alpha_3}} 
%	\cong SL_2\bk{F} \times_{\mu_2}  \Res_{K\slash F}SL_2\bk{F} .
%	\]
Let $Z=X_{\coset{1,2,2}}$ be the center of the Heisenberg group $U$.
We define a polarization $A,A'\leq U$ of the symplectic space $U\rmod Z$ as follows:
\[
\begin{split}
& Z= \set{x_{\coset{1,2,2}}\bk{r}\mvert r\in F} \\
& A = \set{x_{\coset{1,1,0}}\bk{r_1} x_{\coset{1,1,1}}\bk{r_2} x_{\coset{1,1,2}}\bk{r_3} \mvert r_1,r_3\in F,\ r_2\in K} \\
& A' = \set{x_{\coset{0,1,0}}\bk{r_1} x_{\coset{0,1,1}}\bk{r_2} x_{\coset{0,1,2}}\bk{r_3} \mvert r_1,r_3\in F,\ r_2\in K} .
\end{split}
\]
We note that this is a polarization with respect to $Q=L \cdot V$, namely:
\begin{itemize}
\item $U=AA'Z$.
\item $AZ=U\cap V$
\end{itemize}

We fix $\psi:F\to\C^\times$ as a character of $Z\bk{F}\cong F$, namely $\psi\bk{x_{\coset{1,2,2}}\bk{r}} = \psi\bk{r}$, and extend it to a character $\widetilde{\psi}$ of $AZ$.
The Weil representation is defined to be
\[
\bk{\omega_{\widetilde{\psi}}, W_{\widetilde{\psi}}} = \ind_{AZ}^{U}\widetilde{\psi} \cong \Sch\bk{A'}
\]
By Stone-Von Neumann this is the unique smooth irreducible representation of $U$ with central character $\psi$ and hence is independent of the choice of $\widetilde{\psi}$.
Henceforth we fix $\widetilde{\psi}$ to be the extension trivial on $A$.
Our aim is to extend $\omega_{\widetilde{\psi}}$ to $SL_2\bk{F} \overset{\varphi_{\alpha_1}}{\hookrightarrow} M'$.

Consider the 4-dimensional orthogonal space $\mathbb{V}=F\oplus K\oplus F$ defined by the quadratic form
\[
q_K\bk{\vec{r}} = r_1r_3+Tr_{K\slash F}\bk{r_2^2} ,
\]
where $\vec{r}=\bk{r_1,r_2,r_3}\in\mathbb{V}$ with $r_1,r_3\in F$ and $r_2\in K$.

Let $\mathbb{H}=X\oplus Y$ be a two-dimensional symplectic space with the standard isotropic polarization.
In particular, we have $Sp\bk{\mathbb{H}}=SL_2\bk{F}$.

The space $\mathbb{W}=\mathbb{V}\otimes\mathbb{H}$ is equipped with the natural symplectic form and the following polarization
\[
\mathbb{W}=\mathbb{V}\otimes\mathbb{H} = \mathbb{V}\otimes X \oplus \mathbb{V}\otimes Y.
\]
In particular, the dual pair $Sp\bk{\mathbb{H}}\times O\bk{\mathbb{V}}$ splits in $\widetilde{Sp\bk{\mathbb{W}}}$ and the splitting $S_{\mathbb{V},\mathbb{H}}$ is given by \cite[Corollary II.3.3]{KudlaThetaCorr}.	

On the other hand, we have an isomorphism $\mathbb{W}\cong U\rmod Z$ such that $A=\mathbb{V}\otimes X$ and $A'=\mathbb{V}\otimes Y$.
This defines an embedding 
\[
M'\hookrightarrow Sp\bk{\mathbb{W}}
\]
such that $\varphi_{\alpha_1}\bk{SL_2\bk{F}} \subset M'$ is mapped onto  $Sp\bk{\mathbb{H}}$.
Hence, the following diagram is commutative
\[
\xymatrix{
SL_2\bk{F} \ar@{->}[r]^(.6){S_{\mathbb{V},\mathbb{H}}} \ar@{^{(}->}[d]_{\varphi_{\alpha_1}} & \widetilde{Sp\bk{\mathbb{W}}} \ar@{->}[d] \\
M' \ar@{^{(}->}[r] & Sp\bk{\mathbb{W}}\cong Sp\bk{U\rmod Z}
} .
\]
We identify $SL_2\bk{F}$ with its images, $\varphi_{\alpha_1}\bk{SL_2\bk{F}}$ and $Sp\bk{\mathbb{H}}$, under these embeddings.
Using the splitting $S_{\mathbb{V},\mathbb{H}}$, we make $\bk{\omega_{\widetilde{\psi}}, W_{\widetilde{\psi}}}$ into a representation of $SL_2\bk{F}\cdot U$.

The Fourier-Jacobi functor of a representation $\pi$ of	$H$ is defined by
\[
FJ_{\psi}\bk{\pi}=\Hom_{U} \bk{\omega_{\widetilde{\psi}}, \pi_{Z,{\psi}}} ,
\]
which is a representation of $SL_2\bk{F}$.
We note that, according to \cite[Propositon 3.1]{MR1993361}, $FJ_{\psi}$ is independent of $\psi$ hence we denote it by $FJ$.
Our goal is to compute $FJ\bk{I_Q\bk{s}}$.

By \cite[Propositon 3.2]{MR1993361}
\[
FJ\bk{I_{Q}\bk{s}}\otimes \omega_{\widetilde{\psi}} \cong I_Q\bk{s}_{Z,\psi}
\]
Also, by \cite[Propositon 4.2.3]{MR1993361}, we have an $SL_2\bk{F}$-equivariant isomorphism
\[
I_Q\bk{s}_{Z,\psi} \cong I_Q^{w_0}\bk{s}_{Z,\psi},
\]
where $w_0=w_{213421342}$ and
\[
I_Q^{w_0} = \set{f\in I_Q\bk{s} \mvert Supp\bk{f}\subset Qw_0P \text{ and $f$ is compactly supported modulo $Q$}} .
\]

On the other hand, we have:

\begin{Lem}
\label{Lem:WeilRepresentationExtension}
\begin{enumerate}
%	\item The Weil representation $\omega_{\widetilde{\psi}}$ extends to a representation of $Im\bk{\varphi_{\alpha_1}}\ltimes U \cong SL_2\bk{F}\ltimes U$.

\item For any $b=\begin{pmatrix}t&x\\&t^{-1}\end{pmatrix}\in B$ and any $\varphi\in \omega_{\widetilde{\psi}}$ it holds that
\[
Ad\bk{b^{-1}}\circ\omega_{\widetilde{\psi}}\bk{b} \varphi = \chi_K\bk{t} \FNorm{t}^2 \varphi = \chi_K\bk{t}\modf{B}\bk{t} \varphi .
\]

\item
Let $B$ denote the Borel subgroup of $SL_2\bk{F}$ and let $\modf{B}$ denote its modular quasi-character.
There is an $SL_2\bk{F}\cdot U$-equivariant isomorphism
\[
I_B\bk{\chi_K,3s}\otimes \omega_{\widetilde{\psi}} \cong 
\ind_{B \cdot AZ}^{SL_2\bk{F}\cdot U} \coset{\modf{B}^{3s}\otimes\widetilde{\psi}}.
\]
\end{enumerate}
\end{Lem}

\begin{proof}	
\begin{enumerate}	
\item
Applying \cite[Proposition II.4.3]{KudlaThetaCorr} it follows that the Borel subgroup $B$ of $SL_2\bk{F}$ acts on $\omega_{\widetilde{\psi}}$ by
\[
\omega_{\widetilde{\psi}}\bk{\begin{pmatrix}t&x\\&t^{-1}\end{pmatrix}} \varphi\bk{v} = \chi_{K}\bk{t} \FNorm{t}^2 \psi\bk{\frac{x}{2}\left\Vert v\right\Vert^2} \varphi\bk{tv}
\]
where we consider $\varphi\in W_{\widetilde{\psi}}$ as a function on $\mathbb{V}\cong\mathbb{V}\otimes Y = A'$.
It follows that
\[
\coset{Ad\bk{b^{-1}}\circ\omega_{\widetilde{\psi}}\bk{b}}\varphi\bk{v} = \chi_{K}\bk{t}\FNorm{t}^2 \varphi\bk{v} .
\]

In particular, $Ad\bk{b^{-1}}\circ\omega_{\widetilde{\psi}}\bk{b}$ acts on $W_{\widetilde{\psi}}$ by $\chi_K\modf{B}=\chi_{K}\modf{1}^{1/2}$.

\item
This is similar to the second part of the proof of \cite[Theorem 4.3.1]{MR1993361}.
For $f\in I_B\bk{\chi_K,3s}$ and $\varphi\in\omega_{\widetilde{\psi}}$ let
\[
F\bk{mu} = f\bk{m} \bk{\omega_{\widetilde{\psi}}\bk{m}\varphi} \bk{mum^{-1}} \quad \forall \ m\in M,\ u\in U\bk{F} .
\]
It follows from the previous item that $F\in \ind_{B \cdot AZ}^{SL_2\bk{F}\cdot U} \coset{\modf{B}^{3s}\otimes\widetilde{\psi}}$.
The bijectivity follows similarly to that in \cite[Theorem 4.3.1]{MR1993361}.	
\end{enumerate}

\end{proof}

It follows from item (2) of \Cref{Lem:WeilRepresentationExtension} that 
\[
I_Q^{w_0}\bk{s}_{Z,\psi} \cong 
\ind_{\bk{M'\cap Q} \cdot AZ}^{M'\cdot U} \coset{\bk{\modf{Q}^{s}\res{M'\cap Q}}\otimes\widetilde{\psi}} .
%	\ind_{\bk{M'\cap Q} \cdot AZ}^{M'\cdot U} \coset{\modf{B}^{3s}\otimes\widetilde{\psi}}
\]
This isomorphism is realized via the map sending $f\in I_Q^{w_0}\bk{s}_{Z,\psi}$ to
\[
F\bk{g} = \intl_Z f\bk{w_0gz} \overline{\psi\bk{z}} \ dz \quad \forall g\in H^J .
\]
We note that:
\begin{itemize}
\item
$\modf{Q}\bk{h_{\alpha_1}\bk{t}} = \FNorm{t}^6 = \modf{B}^3\bk{h_{\alpha_1}\bk{t}}$.
\item
As representations of $SL_2\bk{F}\cdot U$:
\[
\ind_{\bk{M'\cap Q} \cdot AZ}^{M'\cdot U} \coset{\bk{\modf{Q}^{s}\res{M'\cap Q}}\otimes\widetilde{\psi}} \cong \ind_{B \cdot AZ}^{SL_2\bk{F}\cdot U} \coset{\modf{B}^{3s}\otimes\widetilde{\psi}} .
\]
\end{itemize}

We proved an analogue of \cite[Theorem 4.3.1]{MR1993361}:

\begin{Prop}
$FJ\bk{I_{Q}\bk{s}}=I_{B}\bk{\chi_{K},3s}$.
\end{Prop}

A similar argument to that of \cite[Section 5.1]{MR1993361} shows that the length of $I_{Q}\bk{1/6}$ equals the length of $FJ\bk{I_Q\bk{1/6}}$.
We sketch the proof for the convenience of the reader.

Since the Fourier-Jacobi functor is exact, it follows that the length, $\widetilde{\ell}$, of $FJ\bk{I_Q\bk{s}}$ is at most the length, $\ell$, of $I_{Q}\bk{s}$.
If $\widetilde{\ell}<\ell$ then $FJ$ annihilates $\ell-\widetilde{\ell}$ factors in the composition series of $I_Q\bk{s}$.
\begin{Prop}
\label{Prop:kernelofFJ}
For a smooth irreducible representation $\bk{\sigma,V}$ of $H$ the following are equivalent:
\begin{itemize}
\item $FJ\bk{\sigma}=\bk{0}$.
\item $\sigma$ is trivial.
\end{itemize}
\end{Prop}

Let us show how \Cref{Prop:kernelofFJ} implies the irreducibility of $I_Q\bk{1/6}$.
Indeed, if $FJ\bk{\sigma}=\bk{0}$ then $\sigma$ is trivial.
On the other hand, $I_Q\bk{1/6}$ admits a unique irreducible spherical subquotient which, by \Cref{I:QP:structure} parts (1) and (3), is isomorphic $\Pi^0$.
Since $\Pi^0$ is the Langlands quotient of a standard model, different from $\Ind_{B_H}^H \modf{B_H}^{1/2}$, it is not isomorphic to the trivial representation due to the uniqueness of the model.

Since $FJ\bk{I_Q\bk{1/6}}=I_{B}\bk{\chi_K,1/2}$ is irreducible we conclude that $I_Q\bk{1/6}$ is irreducible.

\begin{proof}[Proof of \Cref{Prop:kernelofFJ}]
The proof is similar to that of \cite[Corollary 6.1.4]{MR1993361}.

First we prove that if $\sigma$ is trivial then $FJ\bk{\sigma}=\bk{0}$.
Plugging non-trivial $\psi$ in the definition of the Fourier-Jacobi functor yields $\sigma_{Z,\psi}=\bk{0}$ and hence $FJ\bk{\sigma}=\bk{0}$.

On the other hand, assume that $FJ\bk{\sigma}=\bk{0}$.
It follows from the assumption that $\sigma_{Z,\psi}=\bk{0}$ and hence, by \cite{MR0425030}, $Z$ acts trivially on $\sigma$.
For any root $\alpha$ defined over $F$ there exists $w\in W$ such that $w\cdot \coset{1,2,2}=\alpha$ and hence for any $r\in F$ and $v\in V$ it holds that
\begin{align*}
\sigma\bk{x_{\alpha}\bk{r}} v 
& = \sigma\bk{w x_{\coset{1,2,2}}\bk{r} w^{-1}} v \\
& = \sigma\bk{w} \circ \sigma\bk{x_{\coset{1,2,2}}\bk{r}} \circ \sigma\bk{w^{-1}} v \\
& = \sigma\bk{w} \circ \sigma\bk{w^{-1}} v = v .
\end{align*}
We conclude that for any root defined over $F$, the associated one-parametric subgroup acts trivially in $\sigma$.

We note that
\[
x_{\coset{1,1,1}}\bk{k} = \coset{x_{\coset{1,2,2}}\bk{1},x_{-\coset{0,1,1}}\bk{k}} \cdot x_{\coset{1,0,0}}\bk{-Tr_{K\rmod F}\bk{k}} .
\]
By the above $x_{\coset{1,2,2}}\bk{1}$ and $x_{\coset{1,0,0}}\bk{-Tr_{K\rmod F}\bk{k}}$ acts trivially on $\sigma$ and hence so does the right hand side of this equation.
It follows that $x_{\coset{1,1,1}}\bk{k}$ acts trivially on $\sigma$.
Since all roots defined over $K$ are in the same $W_H$-orbit the discussion above implies that $X_\alpha$ acts trivially on $\sigma$ for any root defined over $K$.

In conclusion, $X_\alpha$ acts trivially for any root $\alpha$.
Since $H$ is a simply-connected group, it is generated by these one-parametric subgroups.
It follows that $H$ acts trivially on $\sigma$.
Namely, $\sigma$ is the trivial representation of $H$.
\end{proof}

\subsection{Eisenstein Series}
For the rest of this section, we retain the assumption that $F$ is a global number field.
For any field extension $K$ of $F$ we denote by $\zfun_K\bk{s}$ the completed zeta function of $K$ normalized, as in \cite{MR1174424}, so that
\begin{equation}
\zfun_K\bk{s} = \zfun_K\bk{1-s} .
\end{equation}

The Eisenstein series associated to an 
induced representation $I_Q(s)$ 
is an operator that maps every flat section $f(\cdot,s)$ 
to 
\[
\E_Q(f,g,s)=\sum_{\gamma\in Q(F)\backslash H(F)} f(\gamma g,s)
\]
The series converges for $Re\bk{s}$ large and admits a meromorphic 
continuation. The Eisenstein series $\E_P\bk{f,s,g}$  associated
to $I_P(s)$ is defined similarly. 

Our goal is to study the behavior of $\E_Q(f,g,s)$ at $s=1/6$ for various $K$.
This is a standard computation.
Let us first recall the general strategy, for a further discussion and study of the poles of $\E_P$ consult \cite{SegalEisen}.

We recall from \Cref{Subsec:Notations} the injection
\[
I_Q\bk{s} \hookrightarrow I_B\bk{\chi_s^Q} .
\]

The constant term along $N_H$ of $\E_{Q}\bk{s,f,g}$ is given by
\[
\E_{Q}\bk{s,f,g}_{N_H} =
\integral{N_H} \E_{Q}\bk{s,f,ng} \ dn \quad \forall g\in T\bk{\A} .
\]
By a standard computation, as in \cite{MR1469105}, it holds that
\begin{equation}
\E_{Q}\bk{s,f,g}_{N_H} = 
\suml_{w\in W\bk{L,T}} M_w(s)\bk{f_s}\bk{g} \quad \forall g\in T\bk{\A} ,
\end{equation}
where
\[
W\bk{L,T} = \set{w\in W\mvert w\cdot\alpha_1<0,\ w\cdot \alpha_i>0, \forall i\neq 1} 
\]
is the set of the shortest representatives of $Q\lmod H\rmod B_H$.
We recall
\begin{Thm}
The degenerate Eisenstein series $\E_{Q}\bk{s,f,g}$ admits a pole at $s=1/6$ if and only if its constant term $\E_{Q}\bk{s,f,g}_{N_H}$ admits a pole of the same order at $s=1/6$.
\end{Thm}

\begin{Thm}[The Gindikin-Karpelevich Formula]
\[
M_w\bk{\widetilde{f}_s^0} = J\bk{w,s} \widetilde{f}_{w.s}^0,
\]
where
\[
J\bk{w,s} = \prodl_{\alpha>0,\ w^{-1}\alpha<0} 
\frac{\zfun_{F_\alpha}\bk{\gen{\chi_s^Q,\check{\alpha}}}}{\zfun_{F_\alpha}\bk{\gen{\chi_s^Q,\check{\alpha}}+1}} .
\]

\end{Thm}

We also note the following corollary to \cite[Lemma 1.5]{MR1174424} and \cite[Proposition 6.3]{MR944102}.
\begin{Cor}
	\label{Cor: Keys-Shahidi}
	Let $\alpha$ be a simple root and $w_\alpha$ the associated simple reflection and fix $w\in W$.
	Further assume that
	\[
	w_\alpha^{-1}\cdot \coset{w^{-1}\cdot \chi_{s_0}^Q} =
	w^{-1}\cdot \chi_{s_0}^Q .
	\]
	Then $M_{w_\alpha}\bk{w^{-1}\cdot\chi_s^Q}$ is holomorphic at $s_0$.
	Furthermore, $M_{w_\alpha}\bk{w^{-1}\cdot\chi_{s_0}^Q}$ acts as a scalar multiplication by $-1$ on $I_B\bk{w^{-1}\cdot\chi_{s_0}^Q}$.
\end{Cor}

In this subsection we prove:
\begin{Prop} 
\label{Prop:PolesofE_Q}
	\begin{enumerate}
		\item
		Let $K$ be a field.
		Then $\E_{Q}(\cdot,g,s)$ is holomorphic at $s=1/6$.
		
		\item 
		Let $K=F\times F$.
		Then $\E_{Q}(\cdot,g,s)$ has a simple pole at  $s=1/6$.
		The pole is attained by the spherical function.
		Furthermore, the spherical residual representation
		\[
		\Pi = \set{\lim\limits_{s\to1/6} \bk{s-1/6} \E_{Q}(f_s,g,s) \mvert f_s \in I_{Q}^0\bk{1/6} }
		\]
		is isomorphic to the minimal representation.
	\end{enumerate}
\end{Prop}

\begin{proof} 
	\begin{enumerate}	
		\item
		Let $K$ be a field. 
		The following table describes the Gindikin-Karpelevich factor associated to each $w\in W\bk{L,T}$ 
		and the relevant exponents $w^{-1}\cdot\chi_s^Q$.
		\begin{center}
			\begin{table}[H]
				\begin{tabular}{|c|c|c|c|}
					\hline
					$w\in W\bk{L,T}$ & $J\bk{w,s}$ & Order at $1/6$ & $w^{-1}\cdot\chi_s^Q\bk{t_1,t_2,t_3,t_4}$
					\\ \hline
					$1$ & $1$ & $0$ & $\frac{\FNorm{t_1}_F^3}{\FNorm{t_2}_F \FNorm{t_3}_K}$ \\ \hline
					$w\coset{1}$ & $\frac{\zfun_F\bk{6s+2}}{\zfun_F\bk{6s+3}}$ & $0$ & $\frac{\FNorm{t_2}_F^3}{\FNorm{t_1}^3_F \FNorm{t_3}_K}$ \\ \hline
					$w\coset{12}$ & $\frac{\zfun_F\bk{6s+1}}{\zfun_F\bk{6s+3}}$ & $0$ & $\frac{\FNorm{t_3}_K}{\FNorm{t_1}_F \FNorm{t_1}_F^2}$ \\ \hline
					$w\coset{123}$ & $\frac{\zfun_F\bk{6s+1}}{\zfun_F\bk{6s+3}}\frac{\zfun_K\bk{6s}}{\zfun_K\bk{6s+1}}$ & $1$ & $\frac{1}{\FNorm{t_1}_F \FNorm{t_3}_K}$ \\ \hline
					$w\coset{1232}$ & $\frac{\zfun_F\bk{6s+1}}{\zfun_F\bk{6s+3}} \frac{\zfun_F\bk{6s-1}}{\zfun_F\bk{6s}} \frac{\zfun_K\bk{6s}}{\zfun_K\bk{6s+1}}$ & $1$ & $\frac{1}{\FNorm{t_1}_F \FNorm{t_3}_K}$ \\ \hline
					$w\coset{12321}$ & $\frac{\zfun_F\bk{6s+1}}{\zfun_F\bk{6s+3}} \frac{\zfun_F\bk{6s-2}}{\zfun_F\bk{6s}} \frac{\zfun_K\bk{6s}}{\zfun_K\bk{6s+1}}$ & $0$ & $\frac{\FNorm{t_1}_F}{\FNorm{t_2}_F \FNorm{t_3}_K}$ \\ \hline
				\end{tabular}
				%\caption{}
			\end{table}
		\end{center}
		
		A pole of order $1$ is attained for $M_{w\coset{123}}$ and $M_{w\coset{1232}}$.
		By \Cref{Cor: Keys-Shahidi} $M_{w\coset{2}}$ acts as a scalar multiplication by $-1$ on $I_{B_H}\bk{w\coset{123}^{-1}\cdot \chi_{1/6}^Q}$ and hence
		\[
		\lim_{s\to 1/6} \bk{s-1/6} M_{w\coset{1232}} = - \lim_{s\to 1/6} \bk{s-1/6} M_{w\coset{123}} .
		\]
		In particular,
		the  coefficient of $(s-1/6)^{-1}$ in the Laurent expansion 
		of $\E_{Q}\bk{s,f,g}$ is $0$ which proves holomorphicity of $\E_{Q}\bk{s,f,g}$ at $1/6$.
		Note that some  exponents (e.g. of $w=e$) are non-negative and hence the leading term is not square-integrable representation.
				
		\item
		Now assume that $K=F\times F$.
		The following table describes the Gindikin-Karpelevich factors
		associated to each 
		$w\in W\bk{L,T}$ and the relevant exponent $w^{-1}\cdot\chi_s^Q$.
		\begin{center}
			\begin{table}[H]
				\begin{tabular}{|c|c|c|c|}
					\hline
					$w\in W\bk{L,T}$ & $J\bk{w,s}$ & Order of pole at $1/6$ & $w^{-1}\cdot\chi_s^Q\bk{t_1,t_2,t_3,t_4}$
					\\ \hline
					$1$ & $1$ & $0$ & $\frac{\FNorm{t_1}^3}{\FNorm{t_2t_3t_4}}$ \\ \hline
					$w\coset{1}$ & $\frac{\zfun_F\bk{6s+2}}{\zfun_F\bk{6s+3}}$ & $0$ & $\frac{\FNorm{t_2}^2}{\FNorm{t_1^3t_3t_4}}$ \\ \hline
					$w\coset{12}$ & $\frac{\zfun_F\bk{6s+1}}{\zfun_F\bk{6s+3}}$ & $0$ & $\frac{\FNorm{t_3t_4}}{\FNorm{t_1t_2^2}}$ \\ \hline
					$w\coset{123}$ & $\frac{\zfun_F\bk{6s}}{\zfun_F\bk{6s+3}}$ & $1$ & $\frac{\FNorm{t_4}}{\FNorm{t_1t_2t_3}}$ \\ \hline
					$w\coset{124}$ & $\frac{\zfun_F\bk{6s}}{\zfun_F\bk{6s+3}}$ & $1$ & $\frac{\FNorm{t_3}}{\FNorm{t_1t_2t_4}}$ \\ \hline
					$w\coset{1234}$ & $\frac{\zfun_F\bk{6s}^2}{\zfun_F\bk{6s+3}\zfun_F\bk{6s+1}}$ & $2$ & $\frac{1}{\FNorm{t_1t_3t_4}}$ \\ \hline
					$w\coset{12342}$ & $\frac{\zfun_F\bk{6s}\zfun_F\bk{6s-1}}{\zfun_F\bk{6s+3}\zfun_F\bk{6s+1}}$ & $2$ & $\frac{1}{\FNorm{t_1t_3t_4}}$ \\ \hline
					$w\coset{123421}$ & $\frac{\zfun_F\bk{6s}\zfun_F\bk{6s-2}}{\zfun_F\bk{6s+3}\zfun_F\bk{6s+1}}$ & $1$ & $\frac{\FNorm{t_1}}{\FNorm{t_2t_3t_4}}$ \\ \hline
				\end{tabular}
			\end{table}
		\end{center}

		A pole of order $2$ at $s=1/6$ is attained by $M_{w\coset{1234}}$ and $M_{w\coset{12342}}$.
		By \Cref{Cor: Keys-Shahidi} $M_{w\coset{2}}$ acts as a scalar multiplication by $-1$ on $I_B\bk{w\coset{1234}^{-1}\cdot \chi_{1/6}^Q}$
		\[
		\lim_{s\to1/6} \bk{s-1/6}^2 M_{w\coset{12342}} = - 
		\lim_{s\to 1/6} \bk{s-1/6}^2 M_{w\coset{1234}} .
		\]
		In particular, $\E_{Q}\bk{s,f^0,g}$ has at most simple pole at $s=1/6$.
		The pole is attained since $M_{w\coset{123}}$ contributes a simple pole and the exponent of $w\coset{123}^{-1}\cdot \chi_{1/6}^Q$ cannot be canceled by
		the other terms.
		In turn, a simple pole of $w\coset{123}^{-1}\cdot \chi_{1/6}^Q$ is attained by the spherical vector.
		
		We now consider the spherical residual representation.		
		It follows from \Cref{I:QP:structure} that for $\nu\nmid\infty$ the spherical quotient of $I_{Q_\nu}^0\bk{1/6}$ is the minimal representation.
		In fact, this holds for any $\nu$.
		We recall from \cite{MR1918673} that the minimal representation is the unique irreducible quotient of $I_{P_\nu}\bk{3/10}$.
		On the other hand, $I_{Q}^0\bk{1/6}$ is a quotient of $I_{P}\bk{3/10}$ and hence it also has the minimal representation as its unique irreducible quotient.
		\[
		\xymatrix{
			I_{P_\nu}\bk{3/10} \ar@{->>}[r] \ar@{->>}[d] &
			I_{Q_\nu}^0\bk{1/6} \ar@{-->}[dl] \\
			\Pi_{min, \nu}
		} .
		\]		
		The exponents of the terms that contribute the pole of order $1$ are all negative and hence, by Langlands' criterion, the residual representation is square integrable.
		By square integrability we may write
		\[
		\Pi = \oplus_i \Pi_i,
		\]
		where $\Pi_i=\otimes_\nu \Pi_{i,\nu}$ are irreducible quotients of $I_{Q}^0\bk{1/6}$.
		In particular, for each $i$ and $\nu$ it holds that $\Pi_{i,\nu}$ is an irreducible quotient of $I_{Q_\nu}^0\bk{1/6}$.
		From the discussion above it follows that $\Pi=\Pi_{min}$.
		
	\end{enumerate}
	
\end{proof}

In the next section we shall establish a relation between the leading terms of the Eisenstein series of $\E_Q(f,g,s)$ and the Eisenstein series $\E_P(f,g,s)$ that governs the analytic behavior of the standard $\Lfun$-function.
The analytic behavior of $\E_P(f,g,s)$ at $s=3/10$ is detailed in \Cref{poles:Eisen_P}.
The relation will be the key part of the proof of the main theorem.
Let us note that this type of identities exists in much more general situation.

\subsection{Siegel-Weil Identity}
In this subsection we prove a Siegel-Weil like formula.
We start by considering an identity between the leading terms of the Eisenstein series $\E_P$ and $\E_Q$ for the spherical vectors.
\begin{Prop}
	\label{SW}
	Let $f^0\in I_P\bk{s}$ be a spherical section.
	\begin{enumerate}
		\item
		If $K$ is a field then
		\[
		\coset{\bk{s-3/10}\E_{P}\bk{s,f^0_{s},g}}\res{s=3/10} =
		\frac{R}{\zfun_F\bk{2}} \E_{Q}\bk{1/6,\widetilde{f}^0_{1/6},g} .
		\]
		
		\item 
		If $K=F\times F$ then
		\[
		\coset{\bk{s-3/10}^2\E_{P}\bk{s,f^0_{s},g}}\res{s=3/10}
		=
		\frac{R}{\zfun_F\bk{2}} 
		\coset{\bk{6s-1} \E_{Q}\bk{s,\widetilde{f}^0_{s},g}} \res{s=1/6} .
		\]
	\end{enumerate}
\end{Prop}

\begin{Remark}
	Note that the spherical vectors in the LHS and the RHS are in two different degenerate principal series representations.
\end{Remark}

The proof is similar to the proof of \cite[Proposition 1.8]{MR1174424}.
Let us define the following normalized spherical Eisenstein series
\begin{equation}
\label{Eq:NormalizedEisensteinSeries}
\E_{B_H}^\sharp\bk{\lambda,g} = 
\coset{\prodl_{\alpha\in\Phi^+} \zfun_{F_\alpha}\bk{\gen{\lambda,\check{\alpha}}+1} l_\alpha^+\bk{\lambda} l_\alpha^-\bk{\lambda}}
\E_{B_H}\bk{\lambda,f^0_\lambda,g} ,
\end{equation}
where
\[
l_\alpha^{\pm}\bk{\lambda} = \gen{\lambda,\check{\alpha}}\pm 1 .
\]

\begin{Prop}
	The normalized Eisenstein series $\E_B^\sharp\bk{\lambda,g}$ is entire and $W_H$-invariant.
\end{Prop}

For the sake of completeness of presentation, we prove this result in the \hyperref[Appendix-EntirenessOfNormalizedSeries]{Appendix}.
%For the sake of completeness of presentation, we prove this result in \Cref{Appendix-EntirenessOfNormalizedSeries}.

\begin{proof}
	We prove this in the case where $K$ is a quadratic extension of $F$.
	The split case follows similarly.
	We consider the normalized Eisenstein series $\E_{B_H}^\sharp\bk{\lambda,g}$ given in \Cref{Eq:NormalizedEisensteinSeries}.

	Using the $W_H$-invariance of $\E_{B_H}^\sharp\bk{\lambda,g}$ we have
	\[
	\E_{B_H}^\sharp\bk{\lambda,g} = \E_{B_H}^\sharp\bk{w_{1}\cdot\lambda,g} .
	\]
	
	We consider two lines in $\mathfrak{a}_\C^\ast$ given by
	\[
	\modf{B}^{1/2} \modf{P}^{s-1/2} = \mu^P_s = \bk{1,5s-3/2,1},\quad 
	\modf{B}^{1/2} \modf{Q}^{s-1/2} = \mu^Q_s = \bk{6s-2,1,1} .
	\]
	It holds that
	\begin{align*}
	& \Ind_B^H\mu^P_s \twoheadrightarrow I_P\bk{s} \\
	& \Ind_B^H\mu^Q_s \twoheadrightarrow I_Q\bk{s} .
	\end{align*}
	Furthermore, it follows that
	\begin{align*}
	&\lim\limits_{s_1,s_3\to 1} \bk{s_1-1} \bk{s_3-1} \E_{B_H}\bk{\lambda_{\overline{s}},f^0,g} = \E_{P}\bk{s_2,f^0,g} \\
	&\lim\limits_{s_2,s_3\to 1}\bk{s_2-1} \bk{s_3-1} \E_{B_H}\bk{\lambda_{\overline{s}},f^0,g} =
	\E_{Q}\bk{s_1,f^0,g}
	\end{align*}
	
	For convenience, let $\kappa_s = w_{1}\cdot\mu_s^Q = \bk{2-6s,6s-1,1}$.
	By $W_H$-invariance, for any $s\in\C$ it holds that
	\[
	\E_{B_H}^\sharp\bk{\kappa_s,g} = \E_{B_H}^\sharp\bk{\mu_s^Q,g} .
	\]
		
	Since $\E_{B_H}^\sharp\bk{\lambda,g}$ is entire, it is continuous in a neighborhood of $\mu^P_{3/10}=\bk{1,0,1}=\kappa_{1/6}$ and $\mu^Q_{1/6}=\bk{-1,1,1}=w_1\cdot\bk{1,0,1}$.
	In particular, since $\mu_{3/10}^P = w_1\cdot \mu_{1/6}^Q$, it follows that
	\[
	\lim\limits_{s\to 3/10} \E_{B_H}^\sharp\bk{\mu^P_s,g} = 
	\E_{B_H}^\sharp\bk{\bk{1,0,1},g} = 
	\lim\limits_{s\to 1/6} \E_{B_H}^\sharp\bk{\kappa_s,g} = 
	\lim\limits_{s\to 1/6} \E_{B_H}^\sharp\bk{\mu^Q_s,g},
	\]

\begin{figure*}[h]
	\begin{tikzpicture}
	\draw[thick,->] (-4,0) -- (4,0) node[anchor=south] {$s_1$};
	\draw[thick,->] (0,-4) -- (0,4) node[anchor=west] {$s_2$};
	
	\foreach \x in {-3,-2,-1,1,2,3,4}
	\draw (\x cm,1pt) -- (\x cm,-1pt) node[anchor=north] {$\x$};
	
	\foreach \y in {-3,-2,-1,1,2,3,4}
	\draw (1pt,\y cm) -- (-1pt,\y cm) node[anchor=east] {$\y$};
	
	\draw[black, thick] (-4,1) -- (4,1) ;
	\draw (-4,1) coordinate [label=above:$\mu^Q$];
	
	\draw[black, thick] (1,-4) -- (1,4) ;
	\draw (1,4) coordinate [label=right:$\mu^P$];
	
	\draw[black, thick] (-3,4) -- (4,-3) ;
	\draw (-3,4) coordinate [label=right:$\kappa$];
	
	\draw (1,0) coordinate [label=above right:$\mu^P_{3/10}$,color=red] (s0);
	\fill [red] (s0) circle (2pt);
	
	\draw (-1,1) coordinate [label=above left:$\mu^Q_{1/6}$,color=red] (s1);
	\fill [red] (s1) circle (2pt);
	
	\draw[black, dashed] (-2,-4) -- (2,4) ;
	\draw (2,4) coordinate [label=right:$H$];	
	\end{tikzpicture} \\
	{\bf F{\scriptsize IGURE}}: The lines $\mu_s^P$, $\mu_s^Q$ and $\kappa_s$ in the plane given by $s_3=1$. The dashed line indicates the hyperplane $H$ fixed by $w_1$.
\end{figure*}
	
	One checks that
	\begin{align*}
	& \lim\limits_{s\to 3/10} \E_{B_H}^\sharp\bk{\mu^P_s,g}
	= - 2^{10} 3^2 R \zfun_F\bk{2}^2 \zfun_K\bk{2}^2 \zfun_F\bk{3} \zfun_K\bk{3} \zfun_F\bk{4}^2 \coset{\bk{s-3/10}^2\E_{P}\bk{s,f^0_{s},g}}\res{s=3/10} \\
	& \lim\limits_{s\to 1/6} \E_{B_H}^\sharp\bk{\mu^Q_s,g}  =
	- 2^{10} 3^2 R^2 \zfun_F\bk{2} \zfun_K\bk{2}^2 \zfun_F\bk{3} \zfun_K\bk{3} \zfun_F\bk{4}^2
	\E_{Q}\bk{1/6,\widetilde{f}^0_{1/6},g}
	\end{align*}
	and hence
	\[
	\coset{\bk{s-3/10}^2\E_{P}\bk{s,f^0_{s},g}}\res{s=3/10} =
	\frac{R}{\zfun_F\bk{2}} \E_{Q}\bk{1/6,\widetilde{f}^0_{1/6},g} .
	\]
	
\end{proof}

The intertwining operator $M\bk{w[2342],\lambda}\res{\mu^P_s}$ admits a simple pole at $s=3/10$, namely at $\bk{1,0,1}$.
We consider the induced intertwining operator
\[
A_{w\coset{2342}} = \Res_{s=3/10} \bk{M\bk{w[2342],\lambda} \res{\mu^P_s}}: I_P\bk{3/10} \to I_Q^0\bk{1/6}.
\]

Since
\[
A_{w\coset{2342}}f_{3/10}^0 =
\begin{cases}
\frac{R\zfun_F\bk{3}\zfun_K\bk{2}}{5\zfun_F\bk{2}\zfun_F\bk{4}\zfun_K\bk{3}}
\widetilde{f}_{1/6}, & K \text{ field} \\
\frac{R\zfun_F\bk{2}}{5\zfun_F\bk{3}\zfun_F\bk{4}}
\widetilde{f}_{1/6},& K=F\times F
\end{cases} ,
\]
it follows that

\begin{Cor}
\label{Cor: Siegel-Weil Identity}
	Given $f\in I_P\bk{3/10}$ we denote by $f_s\in I_P\bk{s}$ the standard section such that $f_{3/10}=f$.
	\begin{enumerate}
		\item
		If $K$ is a field then
		\[
		\coset{\bk{s-3/10}\E_{P}\bk{s,f_{s},g}}\res{s=3/10} =
		\frac{5\zfun_F\bk{4}\zfun_K\bk{3}}{\zfun_F\bk{3}\zfun_K\bk{2}} \E_{Q}\bk{1/6,A_{w\coset{2342}}f,g}.
		\]
		
		\item
		If $K=F\times F$ then
		\[
		\coset{\bk{s-3/10}^2\E_{P}\bk{s,f_s,g}}\res{s=3/10}
		=
		\frac{5\zfun_F\bk{3}\zfun_F\bk{4}}{\zfun_F\bk{2}^2} 
		\Res_{s=1/6}\E_{Q}\bk{s,A_{w\coset{2342}}f,g}.
		\]
	\end{enumerate}
\end{Cor}

\begin{proof}
We consider the case where $K$ is a field for example.
Indeed, it follows from the above that
\[
\coset{\bk{s-3/10}\E_{P}\bk{s,f_{s}^0,g}}\res{s=3/10} =
\frac{5\zfun_F\bk{4}\zfun_K\bk{3}}{\zfun_F\bk{3}\zfun_K\bk{2}} \E_{Q}\bk{1/6,A_{w\coset{2342}}f_s^0,g}.
\]
Since $f_{3/10}^0$ generates $I_P\bk{3/10}$ and $\E_{P}$, $\E_{Q}$ and $A_{w\coset{2342}}$ are $H$-equivariant the claim follows.
\end{proof}

\subsection{Relation Between Eisenstein Series on $Spin_8$ and $SO_8$.}
Let $(V^8_K,q_K)$ be a quadratic space of dimension $8$ 
and discriminant $K$. The group of its automorphisms 
is denoted by $\overline{H}$. It is a quasi-split group over $F$, split over $K$ of type $D_4$. 
There is an isogeny of algebraic groups 
\[
1\rightarrow \mu_2\rightarrow H\rightarrow \overline{H}\rightarrow 1 .
\]

The parabolic subgroup of $\overline{H}$ corresponding to $Q$ is denoted by $\overline{Q}$.

The map $p:H(F)\rightarrow \overline{H}(F)$ induces the isomorphism of vector spaces
\[
p^\ast: I_{\overline{Q}}\bk{s}\rightarrow I_{Q}\bk{s}.
\]

\begin{Lem}
\label{Lem:StructureofRepnofHbar}
Fix a place $\nu$ of $F$.
\begin{enumerate}
\item
If $I_{Q_\nu}\bk{s}$ is an irreducible $H_\nu$-representation then $I_{\overline{Q}_\nu}\bk{s}$ is an irreducible $\overline{H}_\nu$-representation.

\item
If $I_{Q_\nu}\bk{s}$ is generated by a vector $v$ as an $H_\nu$-representation then $I_{\overline{Q}_\nu}\bk{s}$ is generated by $\overline{v}=\bk{p^\ast}^{-1}\bk{v}$ as an $\overline{H}_\nu$-representation.

\item
Let $\nu$ be a finite or real place of $F$.
The representation $I_{\overline{Q}_\nu}\bk{1/6}$ is irreducible when $K_\nu$ is a field.
If $K_\nu=F_\nu\times F_\nu$ then $I_{\overline{Q}_\nu}(1/6)$ has length two.
The unique irreducible quotient is isomorphic to the minimal representation.
\end{enumerate}
\end{Lem}

\begin{proof}
\begin{enumerate}
\item
Assume $\pi$ is a $\overline{H}_\nu$-subrepresentation of $I_{\overline{Q}_\nu}\bk{s}$, then $\pi$ is generated by any $0\neq \overline{v}\in \pi$.
Let $v=p^\ast\bk{\overline{v}}$, by irreducibility $I_{Q_\nu}\bk{s}$ is generated by $v$, namely
\[
I_{Q_\nu}\bk{s} = Span_\C \set{\pi\bk{h} v \mvert h\in H_\nu} .
\]
On the other hand, for any $h\in H_\nu$ we have $\pi\bk{h}v=\pi\bk{p^\ast\bk{h}}\overline{v}$ and hence
\[
I_{Q_\nu}\bk{s} = Span_\C \set{\pi\bk{h} \overline{v} \mvert h\in H_\nu} = 
Span_\C \set{\pi\bk{p^\ast\bk{h}} \overline{v} \mvert h\in H_\nu} \subset \pi .
\]
It follows that $\pi=I_{\overline{Q}_\nu}\bk{s}$.

\item
This is proven similarly.

\item
For real places the claim is stated in \cite[Theorem 3.3.3]{MR1864829}.
For non-Archimedean places this follows from \cref{I:QP:structure}.
\end{enumerate}
\end{proof}

\begin{Remark}
The structure of $I_{\overline{Q}_\nu}\bk{1/6}$ in the complex case is not known to us.
However, as explained at the end of the proof of \Cref{Prop:PolesofE_Q}, the subrepresentation $I_{\overline{Q}_\nu}^0\bk{1/6}$ admits a unique irreducible quotient isomorphic to the minimal representation.
\end{Remark}

Since the canonical inclusion $G_2\hookrightarrow H$ factors through $\overline{H}$, we get for any $g\in G_2$ 
\[
\E_{\overline{Q}}\bk{f,g,s}= \E_{Q}\bk{p^\ast\bk{f},g,s} .
\]

\section{Global Theta Lift}
The goal of this section is to show that (the leading term of) the Eisenstein series $\E_{\overline{Q}}\bk{\cdot,s,h}$ on $\overline{H}_{F\times K}$ is the regularized global theta lift of (the leading term of) the Eisenstein series $\E_B\bk{\cdot,g,s}$ of $SL_2$ at $s=1/2$ associated to the representation $I_B(\chi_K,1/2)$.

Let us recall the set up of the theta lift. 
The group  $\overline{H}$ is the group of isometries of the quadratic space $V^8_K$ of dimension $8$ and discriminant $K$.
We write $V^8_K=V^6_K +\bH$, where $\bH=Span\{e_0,e_0^\ast\}$ is a hyperbolic plane.
Then $\overline{Q}$ is the parabolic subgroup stabilizing $Span\{e_0\}$ and its Levi subgroup is isomorphic to $T_1\times SO\bk{V^6_K}$, where $T_1\simeq \G_m$.
Note that $T_1\cdot Stab_{\overline{H}\bk{\A}}\bk{e_0}=\overline{Q}\bk{\A}$.

The pair $\bk{SL_2,\overline{H}}$ is a dual pair inside $Sp_{16}$.
There is a splitting $SL_2\bk{\A}\times \overline{H}\bk{\A}\rightarrow \widetilde{Sp_{16}}\bk{\A}$ that depends on the form $q_K$.
We denote the pull-back of the Weil representation $\omega_\psi$ to $SL_2\bk{\A}\times \overline{H}\bk{\A}$ by $\omega_{\psi,q_K}$.

The representation $\omega_{\psi,q_K}$ acts via the Schr\"odinger model on the space of Schwartz functions $\Sch\bk{V^8_K\bk{\A}}$.
It can be realized automorphically via
\begin{align*}
& \theta_{\psi,q_K}:\Sch\bk{V^8_K\bk{\A}}\rightarrow \mathcal{A}\bk{SL_2\bk{\A}\times \overline{H}\bk{\A}} \\
& \theta_{\psi,q_K}\bk{\phi}\bk{g,h} = \sum_{v\in V^8_K\bk{F}} \omega_{\psi,q_K}\bk{g,h}\phi\bk{v} .
\end{align*}

For an automorphic form $\varphi\in \mathcal{A}\bk{SL_2}$ and a Schwartz function $\phi\in \Sch\bk{V^8_K\bk{\A}}$, the global theta lift $\theta_{\psi,q_K}\bk{\phi,\varphi}$ is defined by 
\begin{equation}
\label{theta:lift:def}
\theta_{\psi,q_K}\bk{\phi,\varphi} =
\integral{SL_2}\theta_{\psi,q_K}\bk{\phi}\bk{g,h} \varphi\bk{g}\, dg,
\end{equation}
whenever it converges.

In the following discussion we shall omit subscripts and write $\omega$ and $\theta$ instead of $\omega_{\psi,q_K}$ and $\theta_{\psi,q_K}$ when there is no confusion.

\subsection{Regularization of the Lift}

Let $I_B\bk{\chi_K,s}$ be a normalized induction of $SL_2\bk{\A}$.
The associated Eisenstein series $\E_B\bk{\cdot,g,s}$ is holomorphic when $K$ is a field and has 
at most a simple pole when $K=F\times F$.
In the latter case the pole is attained by the spherical section $f^0$ and the residual representation is the trivial representation.

Let $K$ be a field.
For $\varphi=\Eisen\bk{f,g,s}$, with $f\in I_B\bk{\chi_K,s}$, the integral in \Cref{theta:lift:def} does not converge and hence the regularization is required.
Let us recall the details of regularization. 

As a first step, we shall define a $SL_2\bk{\A}\times \overline{H}\bk{\A}$-submodule $\omega^0$ of $\omega$ such that for $\phi\in\omega^0$ the function $\theta\bk{\phi}\bk{g,h}$ is rapidly decreasing as a function of $g\in SL_2\bk{F}\backslash SL_2\bk{\A}$.

The function $\theta\bk{\phi}\bk{g}$ is rapidly decreasing whenever $\omega\bk{g}\phi\bk{0}=0$ for all $g\in SL_2\bk{\A}$.
Fixing an Archimedean place $\nu_0$, define a map
\[
\overline{T}:\omega_{\nu_0}\rightarrow I_{B_{\nu_0}}\bk{\chi_K,3/2}, \quad \overline{T}\bk{\phi}\bk{g}=\omega\bk{g}\phi\bk{0}
\]
and put $\omega_{\nu_0}^0=Ker\bk{\overline{T}}$.
This allows us to define 
\[
\omega^0=\omega^0_{\nu_0}\otimes \bk{\otimes_{\nu\neq \nu_0} \omega_\nu} .
\]
That is obviously an $SL_2\bk{\A}\times \overline{H}\bk{\A}$-module.
Hence, the map
\[
\theta: \omega^0\otimes I_B\bk{\chi_K,s} \rightarrow \mathcal{A}\bk{\overline{H}},
\]
given by 
\[
\theta\bk{\phi,f}\bk{h}=\integral{SL_2}\theta\bk{\phi}\bk{g,h}\E_B\bk{f,g,s}\, dg ,
\]
is well defined.

Recall that the center $\mathcal{Z}_{\nu_0}(sl_2)$ of the universal enveloping algebra $\mathcal{U}_{\nu_0}\bk{\mathfrak{sl}_2}$ is isomorphic to $\C\coset{\Delta}$ where $\Delta$ is the Casimir operator.
The element $\Delta$ acts on $I_{B_{\nu_0}}\bk{\chi_K,s}$ by the constant $s^2-1/4$.
Consider the element $z=\Delta-2\in \mathcal{Z}\bk{\mathfrak{sl}_2}$.
Then $z$ annihilates the representation $I_B\bk{\chi_K,\pm 3/2}$ and acts by a non-zero
constant $P_z\bk{s}$ on any $I_B\bk{\chi_{K_{\nu_0}},s}$ with $s\neq \pm 3/2$.
 
Clearly, $z$ defines an $SL_2\bk{\A}\times \overline{H}\bk{\A}$-equivariant map from $\omega$ to itself.
Moreover, the image is contained in $\omega^0$ since $z$ commutes with $\overline{T}$ and annihilates $I_{B}\bk{\chi_{K_{\nu_0}},3/2}$.

This allows us to extend the  map $\theta$ from $\omega^0\otimes I_B\bk{\chi_K,s}$ to $\omega\otimes I_B(\chi_K,s)$ by 
\[
\theta^{reg}\bk{\phi,f}=\frac{1}{P_z\bk{s}}\integral{SL_2}\theta\bk{z\phi}\bk{g,h}\E_B\bk{f,g,s}\, dg .
\]

The extension is unique for all $s\neq \pm 3/2$.
Otherwise, having two possible extensions $\theta_1$ and $\theta_2$ of $\theta$, we notice that $\theta_1-\theta_2$ vanishes on $\omega^0\otimes I_B\bk{\chi_K,s}$ and hence defines an $SL_2\bk{\A}$-invariant functional on $I_{B}\bk{\chi_{K_{\nu_0}},3/2}\otimes I_{B}\bk{\chi_{K_{\nu_0}},s}$ which must be zero.

Let $K=F\times F$.
The leading term of $\E_B\bk{\chi_K,s}$ is the trivial representation.
The regularization of its theta lift is defined in \cite{MR2262172}.

\subsection{The Regularized Theta Lift and Eisenstein Series}

The next proposition relates the regularized theta lift defined above to the degenerate Eisenstein series $\E_{\overline{Q}}\bk{\cdot,s,g}$ on $\overline{H}_{F\times K}$.

We recall \cite[Theorem 6.8]{MR1469105}:
\begin{Thm}
Let $K=F\times F$, the residual representation of $\Eisen_Q\bk{\cdot,\cdot,s}$ at $s=1/6$ is the minimal representation, namely
\begin{equation}
\set{Res_{s=1/6}\Eisen_Q\bk{f,g,s}\mvert f_s \text{ holomorphic section of } I_Q\bk{s}} = \Pi_{min} .
\end{equation}
\end{Thm}

We now prove an analogue result for the case where $K$ is a field.

\begin{Prop}
Let $K$ be a field and $Re\bk{s}\gg 0$.
For any $\phi\in \omega^0, f \in I_B\bk{\chi_K,s}$ one has 
\begin{equation}
\theta\bk{\phi,f}\bk{h} = \E_Q\bk{F\bk{\phi,f,s},h,s/3} ,
\end{equation}
where 
\[
F\bk{\phi,f,s}\bk{h} = \int\limits_{N\bk{\A}\lmod SL_2\bk{\A}} \omega\bk{g,h}\phi\bk{e_0} f\bk{g,s}\,dg .
\]
\end{Prop}

\begin{proof}
This follows by the standard unfolding technique.
\begin{align*}
& \integral{SL_2} \theta^{16}\bk{\phi}\bk{g,h} \E_B\bk{f,g,s} \,dg \\
& = \int\limits_{B\bk{F}\backslash SL_2\bk{\A}}  \sum_{v\in V_8^K\bk{F}} \omega\bk{g,h}\phi(v) f(g,s) dg \\
& = \int\limits_{T(F)N(\A)\backslash SL_2(\A)} \sum_{v\in V^8_K(F), q_K(v)=0} \omega\bk{g,h}\phi(v) f(g,s) dg .
\end{align*}

The group $\overline{H}\bk{F}$ acts on the set $\set{v\in V^8_K \mvert q_K\bk{v}=0}$.
There are two orbits: the zero orbit and the open orbit represented by $e_0$.
The contribution of the zero orbit vanishes since $\phi\in \omega^0$. 
Thus the integral above equals
\begin{align*}
& \int\limits_{T\bk{F}N\bk{\A}\lmod SL_2\bk{\A}} 
\sum_{\gamma\in \overline{Q}\bk{F}\backslash \overline{H}\bk{F}} \sum_{t_1\in T_1\bk{F}}
\omega\bk{g,t_1\gamma h} \phi\bk{e_0} f\bk{g,s}\, dg \\
& = \int\limits_{T\bk{F}N\bk{\A}\lmod SL_2\bk{\A}} 
\sum_{\gamma\in \overline{Q}\bk{F}\lmod \overline{H}\bk{F}} \sum_{t\in T\bk{F}}
\omega\bk{t^{-1} g,\gamma h}\phi\bk{e_0} f\bk{g,s}\, dg \\
& = \sum_{\gamma\in \overline{Q}\bk{F}\lmod \overline{H}\bk{F}}
\int\limits_{N\bk{\A}\lmod SL_2\bk{\A}} 
\omega\bk{g,\gamma h}\phi\bk{e_0} f\bk{g,s}\, dg=
\sum_{\gamma\in \overline{Q}\bk{F}\lmod \overline{H}\bk{F}} F\bk{\phi,f,s}\bk{\gamma h}
\end{align*}
as required.

It remains to check  that $F\bk{\phi,f,s}$ belongs to $I_{Q}(s/3)$.
Indeed, for $\bk{t,m}\in\overline{L}$ it holds that
\begin{align*}
& F\bk{\phi,f,s}\bk{\bk{t,m}h} =
\int\limits_{N\bk{\A}\lmod SL_2\bk{\A}} \omega\bk{g,\bk{t,m}h}\phi\bk{e_0} f\bk{g}\, dg \\
& = \int\limits_{N\bk{\A}\lmod SL_2\bk{\A}} \omega\bk{g,h}\phi\bk{t^{-1}e_0} f\bk{g}\, dg .
\end{align*}

By the formulas in the Schr\"odinger model one has 
\[
\begin{split}
& \phi\bk{t^{-1}e_0} = \omega\bk{t^{-1},1}\FNorm{t}^{4} \chi_K\bk{t} \phi\bk{e_0} \\
& f\bk{g,s} = \chi_K\bk{t}\FNorm{t}^{1+2s} f\bk{t^{-1}g} \\
& dg=\FNorm{t}^{-2} d\bk{t^{-1}g}
\end{split}
\]
and hence
\[
F\bk{\phi,f,s}\bk{\bk{t,m}g}=\FNorm{t}^{3+2s} F\bk{\phi,f,s}\bk{g} = \delta^{s/3+1/2}_Q\bk{t} F\bk{\phi,f,s}\bk{g}
\]
as required.
\end{proof}

\begin{Cor}
For $Re\bk{s}\gg 0$ and $\phi=\otimes \phi_\nu\in \omega^0, f=\otimes f_\nu\in I_B\bk{\chi_K,s}$ there is a factorization
\[
F\bk{\phi,f,s}\bk{h} = \prodl_\nu F_\nu\bk{\phi_\nu,f_\nu,s}\bk{h_\nu} ,
\]
where 
\[
F_\nu\bk{\phi_\nu,f_\nu,s}\bk{h} = \int\limits_{N\bk{F_\nu}\lmod SL_2\bk{F_\nu}} 
\omega_\nu\bk{g,h} \phi_\nu\bk{e_0} f_\nu\bk{g,s} \, dg .
\]
\end{Cor}

Note that the section $F\bk{\phi,f,s}$ is not standard, but holomorphic for $Re\bk{s}\gg 0$.
We shall define $F\bk{\phi,f,1/2}$ by analytic continuation.

\begin{Prop}
\begin{enumerate}
\item
For any place $\nu$, the map $F_\nu$ is holomorphic and non-zero for $Re\bk{s}>-3/2$. 
\item
Let $K_\nu$ be either a split or unramified quadratic extension of $F_\nu$. We fix the normalized spherical vectors $\phi_\nu^0\in \omega_\nu$ and $f^0_\nu$.
Then $F_\nu\bk{\phi^0_\nu,f^0_\nu,s}= \zeta_\nu\bk{2s+3}$. 
\item
For any place $\nu$ the image of $F_\nu: \omega_\nu\otimes I_B\bk{\chi_{K_\nu},1/2} \rightarrow I_{\overline{Q}_\nu}\bk{1/6}$ contains $I^0_{\overline{Q}_\nu}\bk{1/6}$. 
\end{enumerate}
\end{Prop}

\begin{Remark}
If $\nu$ is an Archimedean field then 
\[
\zeta_\nu\bk{s} = \piece{\pi^{-s/2} \Gamma\bk{s/2},& F_\nu=\R \\ \bk{2\pi}^{1-s}\Gamma\bk{s},& F_\nu=\C}.
\]
\end{Remark}

\begin{proof}
Recall that the section $f\bk{\cdot,s}$ is flat.
In other words the restriction to $K$ does not depend on $s$.
Using the Iwasawa decomposition $SL_2\bk{F_\nu} = N\cdot T\cdot\mathcal{K}$, where $\mathcal{K}$ is the maximal compact subgroup of $SL_2\bk{F_\nu}$, we can write
\[
F_\nu\bk{\phi_\nu,f_\nu,s}=\int\limits_{\mathcal{K}} L\bk{k\phi_\nu,s} f_\nu\bk{k}\, dk, 
\]
where
\[
L\bk{\phi,s} = \int\limits_{T\bk{F_\nu}} \omega\bk{t} \phi\bk{e_0} \chi_K\bk{t} \delta_B^{s+1/2}\bk{t} \delta_B^{-1}\bk{t}\, dt = \int\limits_{F_\nu^\times} \FNorm{t}^{2s+3} \phi\bk{te_0}\, d^\times t .
\]

Obviously, the operator $L\bk{\phi,s}$ is holomorphic and does not vanish for any $\nu$ and $Re\bk{s}>-3/2$. Moreover, $L\bk{\phi_\nu^0,s} = \zeta_\nu\bk{2s+3}$. 

Thus, the operator $F_\nu$ is also holomorphic in that region and if $I_B\bk{\chi_K,s}$ is an unramified representation with $f^0$ is a normalized spherical vector then
\[
F_\nu\bk{\phi_\nu^0,f_\nu^0,s} = \zeta_\nu\bk{2s+3} .
\]

If $K_\nu=F_\nu\times F_\nu$ then the representation $I_B\bk{\chi_{K_\nu},1/2} = I_B\bk{1,1/2}$ is spherical and the image of the spherical data is not zero.
If $K_\nu$ is a field  then $I_{\overline{Q}_\nu}\bk{1/6}$ is irreducible as shown in \Cref{Lem:StructureofRepnofHbar} and hence part $(1)$ implies part $(3)$. 
\end{proof}

It follows that $F\bk{\phi,f,s}$ can be analytically extended with the extension being holomorphic at $s=1/2$ and its image contains $I^0_Q\bk{1/6}$.

\section{See-Saw Diagram}
In this section we prove the \Cref{main}.
So far, we have the theta lifts defined for the dual pair $SL_2\bk{\A}\times \overline{H}\bk{\A}$ as well as $\TSL\bk{\A}\times G_2\bk{\A}$.
They fit into a see-saw pair.
\begin{figure*}[h!]
\[
\xymatrix{
\TSL\times\TSL \ar@{-}[rd] &
\overline{H} \ar@{-}[ld] \\
SL_2 \ar@{^{(}->}[u] &
G_2\times SO\bk{V_K^1} \ar@{_{(}->}[u]
}
\]
%\[
%\begin{diagram}
%\node{\TSL\times\TSL} \arrow{s}
%\node{\overline{H}} \\
%\node{SL_2} \arrow{ne}
%\node{G_2\times SO\bk{V_K^1}} \arrow{n} \arrow{nw}
%\end{diagram}
%\]
{\bf F{\scriptsize IGURE}}: See-saw diagram
\end{figure*}

Using the see-saw duality we are able to prove the following key statement.

\begin{Prop}
Let $\pi$ be a cuspidal representation of $G_2\bk{\A}$ and $I_{\overline{Q}}\bk{1/6}$ be the induced representation of $\overline{H}=\overline{H}_{F\times K}$, where $K$ is a field.

Assume that for some $\varphi\in \pi$ and $f\in I_{\overline{Q}}^0\bk{1/6}$ it holds that
\[
\integral{G_2}\varphi\bk{g} \E_{\overline{Q}}\bk{g,f,1/6}\, dg \neq 0.
\]
Then $RS_\psi\bk{\pi}\neq 0$.
\end{Prop}

\begin{proof}
Since the subrepresentation of $I^0_{\overline{Q}}\bk{1/6}$ is generated by the spherical vector, $I^0_{\overline{Q}}\bk{1/6}$ is contained in the image of $F$.
Hence, there exist $\phi\in \omega^0$ and $f \in I_B\bk{\chi_K,1/2}$ such that 
\[
\integral{G_2} \varphi\bk{h}\,\E_{\overline{Q}}\bk{F\bk{\phi,f,1/2},h,1/6}\, dh\neq 0 .
\]
In particular,
\[
\integral{G_2} \varphi\bk{h} \integral{SL_2}
\theta_{\psi,q_K}\bk{\phi} \bk{g,h} \E_B\bk{f,g,1/2}\, dg\, dh\neq 0
\]
for some choice of data.
Since $\varphi$ is rapidly decreasing on $G_2$ and $\theta(\phi)$ is rapidly decreasing on $SL_2$, the integral converges absolutely and hence we can change the order of integration, so that
\begin{equation}
\label{see:saw}
\integral{SL_2} \integral{G_2} \varphi\bk{h}\theta_{\psi,q_K}\bk{\phi}\bk{g,h} dh \,\E_B\bk{f,g,1/2} \, dg\neq 0
\end{equation}
for some choice of data.

Decompose $\bk{V_K^8, q_K} = \bk{V^7,q} \oplus \bk{V^1_K, q^1_K}$, where $V^1_K$ is a one-dimensional quadratic space of discriminant $K$ and $V^7$ has dimension $7$ and discriminant $1$.
Then
\[
\omega^{16}_{\psi,q_K}\simeq \omega^{14}_{\psi,q^7}\otimes \omega^2_{\psi,q^1_K} .
\]
Thus the function $\phi\in \omega^0$ can be written as $\sum \phi'_i\otimes\phi''_i$, where $\phi_i\in \omega^{14}$ and $\phi''_i\in \omega_K^2$ respectively.
So
\[
\theta_{\psi,q_K}\bk{\phi}\bk{g,h} = \sum \theta_{\psi,q}\bk{\phi'_i} \bk{g,h} \theta_{\psi,q^1_K}\bk{\phi''_i}\bk{g}
\]
and the integral in the LHS of \Cref{see:saw} equals 
\[
\integral{SL_2} \sum_i \bk{\integral{G_2} \varphi\bk{h}
\theta_{\psi,q}\bk{\phi'_i} \bk{g,h} \,dh} \theta_{\psi,q^1_K}\bk{\phi''_i}\bk{g} \E_B\bk{f,g,1/2} \, dg\neq 0.
\]

In particular, for some $i$, the inner integral is not zero which means that $RS_\psi\bk{\pi}\neq 0$ as required.
\end{proof}

Now it is easy to prove \Cref{main}.

If $\Lfun^S\bk{s,\pi,st}$ has a pole at $s=2$ then either the wave front $\widehat{F}\bk{\pi}$ consists of $F\times F\times F$ only or it contains the algebra $F\times K$, where $K$ is a quadratic 
field extension.
In the first case, the pole at $s=2$ is double by \Cref{FC}.
In this case $RS_\psi\bk{\pi}\neq 0$ and in fact, $\pi$ is contained in $\mathcal A_{\chi_0}$, where $\chi_0$ is the trivial character.

In the second case, we conclude that there exist $\varphi\in\pi$ and $f\in I_P\bk{s}$ such that
\begin{equation}
\integral{G_2}\varphi\bk{h} Res_{s=3/10} \E_P\bk{f,h,s} \, dh\neq 0
\end{equation}
or, by \Cref{Cor: Siegel-Weil Identity}, equivalently
\begin{equation}
\integral{G_2}\varphi\bk{h} \E_Q\bk{A_{w[2342]}f,h,1/6} \, dh = 
\integral{G_2}\varphi\bk{h} \E_{\overline{Q}}\bk{p^\ast\bk{A_{w[2342]}f},h,1/6} \, dh
\neq 0 .
\end{equation}

Since $A_{w[2342]}f$ belongs to $I_Q^0\bk{1/6}$, one has $p^\ast\bk{A_{w[2342]}f}\in I^0_{\overline{Q}}\bk{1/6}$ and now the result follows from the previous proposition.

\appendix
\setcounter{secnumdepth}{0}
\section{Appendix. The Normalized Eisenstein Series}
\label{Appendix-EntirenessOfNormalizedSeries}
Assume that $G$ is a quasi-split, simply connected and simple group.
Let $B$ a Borel subgroup of $G$ with maximal torus $T$ containing a maximal split torus $T_S$.
Let $\Phi$ denote the relative root system of $G$ with respect to $\bk{B,T_S}$ with relative Weyl group $W=W\bk{G,T_S}$.
Let $\Phi^{+}$ denote the associated the associated set of positive roots of $G$ with the set of simple roots $\Delta$.
The Weyl group $W$ is generated by the simple reflections $w_\alpha$ with $\alpha\in\Delta$.
For any $\alpha\in\Phi^{+}$ we denote by $F_\alpha$ the field of definition of the root $\alpha$.
Finally, we denote the space $X^\ast_{nr}\bk{T\bk{\A}}$ of unramified characters by $\mathfrak{a}_\C^\ast$.

We consider the Eisenstein series $\E_B\bk{\lambda,f^0_\lambda,g}$ corresponding to the normalized spherical section $f^0_\lambda\in\Ind_B^G\lambda$ (normalized induction).
We recall here that $f^0$ is a standard section.

Let
\[
\E_B^\sharp\bk{\lambda,g} = 
\coset{\prodl_{\alpha\in\Phi^+} \zfun_{F_\alpha}\bk{\gen{\lambda,\check{\alpha}}+1} l_\alpha^+\bk{\lambda} l_\alpha^-\bk{\lambda}}
\E_B\bk{\lambda,f^0_\lambda,g} ,
\]
where
\[
l_\alpha^{\pm}\bk{\lambda} = \gen{\lambda,\check{\alpha}}\pm 1 .
\]

\begin{Prop}
The normalized Eisenstein series $\E_B^\sharp\bk{\lambda,g}$ is entire and $W$-invariant.
\end{Prop}

\begin{proof}
We start by proving that $\E_B^\sharp\bk{\lambda,g}$ is indeed $W$-invariant.
It suffice to prove it for a simple reflection $w_\beta$ for some $\beta\in\Delta$ since $W$ is generated by these reflections.

Indeed, applying the functional equation (\cite[IV.1.10]{MR1361168})
\[
\E_B\bk{\lambda,f^0_\lambda,g} = \E_B\bk{w_\beta\cdot\lambda, M\bk{w_\beta,\lambda}f^0_{\lambda},g},
\]
the fact that
\begin{align*}
& \prodl_{\alpha\in\Phi^+} l_\alpha^+\bk{w_\beta\cdot\lambda} l_\alpha^-\bk{w_\beta\cdot\lambda} =
\prodl_{\alpha\in\Phi^+} l_\alpha^+\bk{\lambda} l_\alpha^-\bk{\lambda} \\
& \prodl_{\alpha\in\Phi^+} \zfun_{F_\alpha}\bk{\gen{w_\beta\cdot\lambda,\check{\alpha}}+1} = 
\frac{\zfun_{F_\beta}\bk{-\gen{\lambda,\check{\beta}}+1}}{\zfun_{F_\beta}\bk{\gen{\lambda,\check{\beta}}+1}} \prodl_{\alpha\in\Phi^+} \zfun_{F_\alpha}\bk{\gen{\lambda,\check{\alpha}}+1}
\end{align*}
and the functional equation
\[
\zfun_{F_\alpha}\bk{s} = \zfun_{F_\alpha}\bk{1-s} \quad \forall \alpha\in\Phi^{+}
\]
yields	
\begin{align*}
& \E_B^\sharp\bk{w_\beta\cdot\lambda,g} 
= \coset{\prodl_{\alpha\in\Phi^+} \zfun_{F_\alpha}\bk{\gen{w_\beta\cdot\lambda,\check{\alpha}}+1} l_\alpha^+\bk{w_\beta\cdot\lambda} l_\alpha^-\bk{w_\beta\cdot\lambda}}
\E_B\bk{w_\beta\cdot\lambda,f^0_{w_\beta\cdot\lambda},g} \\
& = \frac{\zfun_{F_\beta}\bk{-\gen{\lambda,\check{\beta}}+1}}{\zfun_{F_\beta}\bk{\gen{\lambda,\check{\beta}}+1}} \coset{\prodl_{\alpha\in\Phi^+} \zfun_{F_\alpha}\bk{\gen{w_\beta\cdot\lambda,\check{\alpha}}+1} l_\alpha^+\bk{\lambda} l_\alpha^-\bk{\lambda}} \E_B\bk{\lambda,M\bk{w_\beta,w_\beta\cdot\lambda}f^0_{w_\beta\cdot\lambda},g} \\
& = \frac{\zfun_{F_\beta}\bk{-\gen{\lambda,\check{\beta}}+1}}{\zfun_{F_\beta}\bk{\gen{\lambda,\check{\beta}}+1}} \coset{\prodl_{\alpha\in\Phi^+} \zfun_{F_\alpha}\bk{\gen{w_\beta\cdot\lambda,\check{\alpha}}+1} l_\alpha^+\bk{\lambda} l_\alpha^-\bk{\lambda}} 
\frac{\zfun_{F_\beta}\bk{-\gen{\lambda,\check{\beta}}}}{\zfun_{F_\beta}\bk{-\gen{\lambda,\check{\beta}}+1}} \E_B\bk{\lambda,f^0_{\lambda},g} \\
& = \coset{\prodl_{\alpha\in\Phi^+} \zfun_{F_\alpha}\bk{\gen{\lambda,\check{\alpha}}+1} l_\alpha^+\bk{\lambda} l_\alpha^-\bk{\lambda}}
\E_B\bk{\lambda,f^0_\lambda,g} = \E_B^\sharp\bk{\lambda,g} .
\end{align*}

We now turn to prove the $\E_B^\sharp\bk{\lambda,g}$ is entire.
It follows from the general theory of Eisenstein series that $\E_B^\sharp\bk{\lambda,g}$ is entire if and only if its constant term along $B$ is entire.
We recall that
\[
\E_B\bk{\lambda,t}_{T} = \suml_{w\in W} \bk{\prodl_{\stackrel{\alpha\in\Phi^+}{w^{-1}\cdot\alpha\notin\Phi^+}} \frac{\zfun_{F_\alpha}\bk{\gen{\lambda,\check{\alpha}}}}{\zfun_{F_\alpha}\bk{\gen{\lambda,\check{\alpha}}+1}}} f^0_{w^{-1}\cdot\lambda}\bk{t} .
\]
We write
\[
F\bk{\lambda,t} = \E_B^\sharp\bk{\lambda,t}_{T_H} = L\bk{\lambda} \suml_{w\in W} F_w\bk{\lambda} f^0_{w^{-1}\cdot\lambda}\bk{t},
\]
where
\[
L\bk{\lambda} = \prodl_{\alpha\in\Phi^+} l_\alpha^+\bk{\lambda} l_\alpha^-\bk{\lambda}
\]
and
\[
F_w\bk{\lambda} = \bk{\prodl_{\stackrel{\alpha\in\Phi^+}{w^{-1}\cdot\alpha\in\Phi^+}} \zfun_{F_\alpha}\bk{\gen{\lambda,\check{\alpha}}+1}} \bk{\prodl_{\stackrel{\alpha\in\Phi^+}{w^{-1}\cdot\alpha\notin\Phi^+}} \zfun_{F_\alpha}\bk{\gen{\lambda,\check{\alpha}}}} 
\]

While $L\bk{\lambda}$ is entire, $F_w\bk{\lambda,t}$ can possibly have poles along the hyper-planes
\[
H_\alpha^\epsilon = \set{\lambda\in\mathfrak{a}_\C^\ast \mvert \gen{\lambda,\check{\alpha}} = \epsilon}, \quad \alpha\in\Phi^+, \quad \epsilon=-1,0,1 .
\]
We show that these poles cancel, either with each other or with the zeros of $L\bk{\lambda}$.
Let
\[
X = \bigcup_{\stackrel{\epsilon,\epsilon'\in\set{-1,0,1}}{\alpha,\alpha'\in\Phi^{+}, \alpha\neq\alpha'}} \bk{H_{\alpha}^{\epsilon} \cap H_{\alpha'}^{\epsilon'}} ,\qquad
Y = \bigcup_{\stackrel{\epsilon\in\set{-1,0,1}}{\alpha\in\Phi^+}} H_{\alpha}^{\epsilon} .
\]
The set $X$ is a closed subset of $\mathfrak{a}_\C^\ast$ of codimension 2.
We prove that $F\bk{\lambda,t}$ is holomorphic on $\mathfrak{a}_\C^\ast$.

For any $w\in W$ and $\alpha\in\Phi^{+}$ the poles of $F_w\bk{\lambda}$ along $H_\alpha^{\pm 1}\setminus X$ cancel by the zeros of $L\bk{\lambda}$ along these hyper-planes and since
\[
F\bk{\lambda,t} =  \suml_{w\in W} L\bk{\lambda} F_w\bk{\lambda} f^0_{w^{-1}\cdot\lambda}\bk{t}
\]
so is $F\bk{\lambda,t}$.

Fix a simple root $\alpha\in \Delta$.
We prove that $F\bk{\lambda,t}$ is holomorphic along $H_\alpha^0\setminus X$.
This is enough as $F\bk{\lambda,t}$ is $W$-invariant and all short roots are $W$-conjugate and so are all the long roots.
The function $F_w\bk{\lambda}$ admits a simple pole along $H_\alpha^0\setminus X$ for any $w\in W$.

On the other hand, we note that for all $\lambda\notin Y$
\[
F_{w_\alpha w} \bk{\lambda} =  F_w\bk{w_\alpha^{-1}\cdot \lambda} .
\]
On the other hand, one checks that
\[
\Res_{H_\alpha^0} F_w\bk{w_\alpha^{-1}\cdot \lambda} = \Res_{H_\alpha^0} F_w\bk{\lambda} .
\]
It follows that
\[
\Res_{H_\alpha^0} F_w\bk{\lambda} f^0_{w^{-1}\cdot\lambda}\bk{t} = 
- \Res_{H_\alpha^0} F_{\bk{w_\alpha w}}\bk{\lambda} f^0_{\bk{w_\alpha w}^{-1}\cdot\lambda}\bk{t} 
\]
and hence the pole along $H_\alpha^0$ cancels.

We proved that $\E_B^\sharp\bk{\lambda,t}_{T}$ is holomorphic on $\mathfrak{a}_\C^\ast\setminus X$.
By Hartogs' theorem, \cite[Theorem 2.3.2]{MR1045639}, $\E_B^\sharp\bk{\lambda,t}_{T}$ is entire and so is $\E_B^\sharp\bk{\lambda,t}$.
\end{proof}

\bibliographystyle{alpha}
\bibliography{../bib}

\end{document}